\documentclass[a4paper, oneside,11pt]{article}

\usepackage[latin1]{inputenc}
\usepackage[T1]{fontenc}
\usepackage[english]{babel}
\usepackage{amssymb}
\usepackage{amsmath}
\usepackage{amsthm}
\usepackage{graphicx}

\usepackage[all]{xy}
\usepackage{verbatim}
\usepackage{amsmath}
\usepackage{amsthm}
\usepackage{graphicx}
\usepackage{amssymb}
\usepackage{epstopdf}
\usepackage{url}
\usepackage{amsfonts}
\usepackage{todonotes}
\newcommand{\mc}{\mathcal}
\newcommand{\pt}{\partial}

\newcommand{\br}{\mathbb{R}}

\newcommand{\bt}{\mathbb{T}}

\newcommand{\e}{\varepsilon}

\renewcommand{\(}{\left(}
\renewcommand{\)}{\right)}
\renewcommand{\[}{\left[}
\renewcommand{\]}{\right]}

\newcommand{\bb}{\mathbb}

\newtheorem{thm}{Theorem}
\newtheorem{lem}[thm]{Lemma}

\newtheorem{prop}[thm]{Proposition}
\newtheorem{defi}[thm]{Definition}
\newtheorem{remark}[thm]{Remark}

\newtheorem{hyp}[thm]{Assumption}
\def\be{\begin{equation}}
\def\ee{\end{equation}}
\def\bea{\begin{eqnarray}}
\def\eea{\end{eqnarray}}

\newcommand{\RR}{\mathbb{R}}
\newcommand{\TT}{\mathbb{T}}

\usepackage{enumerate}

\numberwithin{thm}{section}
\numberwithin{equation}{section}

\usepackage[hyperindex,breaklinks]{hyperref}
\usepackage{mathtools}
\mathtoolsset{showonlyrefs}

\newcommand*\di{\mathop{}\!\mathrm{d}}

\title{Global well-posedness for the Vlasov-Poisson system with\\ massless electrons in the 3-dimensional torus}

\author{
Megan Griffin-Pickering 
  \thanks{Durham University, Department of Mathematical Sciences, Lower Mountjoy, Stockton Road, Durham DH1 3LE, UK.  Email: \textsf{megan.k.griffin-pickering@durham.ac.uk}}
  \and
Mikaela Iacobelli
  \thanks{ETH Zurich, Department of Mathematics, Ramistrasse 101, 8092 Zurich, Switzerland.\newline Email: \textsf{mikaela.iacobelli@math.ethz.ch}}
}

\begin{document}

\maketitle

\begin{abstract}
The Vlasov-Poisson system with massless electrons (VPME) is widely used in plasma physics to model the evolution of ions in a plasma. It differs from the Vlasov-Poisson system (VP) for electrons in that the Poisson coupling has an exponential nonlinearity that creates several mathematical difficulties. In particular, while global well-posedness in 3D is well understood in the electron case, this problem remained completely open for the ion model with massless electrons. The aim of this paper is to fill this gap by proving uniqueness for VPME in the class of solutions with bounded density, and global existence of solutions with bounded density for a general class of initial data, generalising all the previous results known for VP.
\end{abstract}

\section{Introduction}

In this article, we study a kinetic model for the ions in a dilute plasma.
A plasma is a ionised gas, in which gas particles have dissociated into ions and electrons. The ions are positively charged, while the electrons are negatively charged and have a much smaller mass than the ions. 

To model a fully ionised plasma one should consider a coupled system involving both ions and electrons. However, since the masses of the two species have very different orders of magnitude, there is a separation between the timescales on which each species evolves. From the point of view of the electrons, the ions are very heavy and so slow-moving. For this reason, it is common to assume that the ions are stationary over the interval of observation.  If magnetic effects are also neglected, this leads to the well-known \textit{Vlasov-Poisson system}. This system is often considered either on the whole space or on the flat torus, and in this paper we will focus on the latter case:
\begin{equation}
\label{vp}
(VP):= \left\{ \begin{array}{ccc}\pt_t f+v\cdot \nabla_x f+ E\cdot \nabla_v f=0,  \\
E=-\nabla U, \\
\Delta U=1- \int_{\br^d} f\, dv= 1- \rho_f,\\
f\vert_{t=0}=f_{0}\ge0,\ \  \int_{\bt^d \times \br^d} f_{0}\,dx\,dv=1.
\end{array} \right.
\end{equation}
Here the unknown $f = f(t,x,v)$ is a probability density describing the distribution of electrons at time $t,$ position $x$, and velocity $v,$ with $(x,v) \in \bt^d\times\br^d$. The Vlasov-Poisson system thus describes the evolution of the electrons under the influence of the electrostatic potential $U$ induced by the charge distribution of the entire plasma. This encodes the fact that the long-range effect of the potential is dominant over the effect of collisions between the electrons, and describes the electrostatic regime in which magnetic effects may be neglected.

\noindent Observe that the electric field $E$ can be represented in the form
\be
E = - \nabla G \ast \rho_f,
\ee
where $G$ is the Green kernel of the negative Laplacian on the torus, that is
\be \label{def:G}
- \Delta G = \delta_0 - 1.
\ee
The Coulomb kernel $K = -\nabla G$ has a strong singularity of order $|x|^{-(d-1)}$ at $x=0$, and its derivative $\nabla^2 G$ thus has a non-integrable singularity.
This is the reason why the study of the well-posedness theory of the Vlasov-Poisson system is mathematically challenging. Global-in-time classical solutions have been constructed under various conditions on the initial data (see for example \cite{Bardos-Degond, BR, Lions-Perthame, Pfaffelmoser, Schaeffer, Ukai-Okabe}), while global-in-time weak solutions were constructed in \cite{Arsenev} and \cite{Horst-Hunze} for $L^p$ initial data (see also \cite{BD85, BDG86}). However, uniqueness is not known to hold in general for weak solutions. An important contribution to the uniqueness theory was made by Loeper \cite{Loep}, who showed uniqueness for solutions of \eqref{vp} with bounded density.

\begin{paragraph}{The Vlasov-Poisson system with massless electrons, or Vlasov-Poisson system for ions.}

In this paper, we will consider a different model for plasma where the unknown is the repartition function of ions, instead of the one for the electrons. This model has been introduced to take into account the dramatically different order of magnitude between ions and electrons. Indeed, the electrons move much more quickly than the ions, and therefore undergo collisions much more frequently. Consequently, they approach thermodynamic equilibrium rapidly. One therefore assumes that the electrons are thermalised, obeying a Maxwell-Boltzmann law. Then their spatial density is given by $e^U$ and the induced electric field obeys a nonlinear Poisson equation with exponential nonlinearity.
This leads to the \textit{Vlasov-Poisson system with massless electrons} (VPME):
\begin{equation}
\label{vpme}
(VPME):= \left\{ \begin{array}{ccc}\pt_t f+v\cdot \nabla_x f+ E\cdot \nabla_v f=0,  \\
E=-\nabla U, \\
\Delta U=e^{U}- \int_{\br^d} f\, dv=e^{U}- \rho_f,\\
f\vert_{t=0}=f_{0}\ge0,\ \  \int_{\bt^d \times \br^d} f_{0}\,dx\,dv=1.
\end{array} \right.
\end{equation}

This system can be derived, at least formally, as the limiting regime of a coupled system of ions and electrons, in the \textit{massless electrons} limit where the ratio $m_e/m_i$ between the electron and ion masses tends to zero. Bardos, Golse, Nguyen and Sentis \cite{BGNS18} studied this limit, considering coupled systems of the form
\be \label{eq:VP-coupled}
\begin{cases}
\partial_t f_i + v \cdot \nabla_x f_i + \frac{q_i}{m_i} E \cdot \nabla_v f_i = 0, \\
\partial_t f_e + v \cdot \nabla_x f_e + \frac{q_e}{m_e} E \cdot \nabla_v f_e = C(m_e) Q(f_e), \\
\nabla_x \times E = 0, \quad
\epsilon_0 \nabla_x \cdot E = q_i \rho[f_i] + q_e \rho[f_e] ,
\end{cases}
\ee
where $Q$ denotes a collision operator such as a BGK or Boltzmann operator. Under suitable assumptions on $C(m_e)$ and the existence of sufficiently regular solutions for the coupled system, they derive that, in the limit as $m_e/m_i$ tends to zero, the electrons indeed take on a Maxwell-Boltzmann distribution, and the system \eqref{eq:VP-coupled} converges to a system similar to \eqref{vpme}, but with a time-dependent electron temperature. In a similar vein, Bouchut and Dolbeault \cite{Bouchut-Dolbeault95} studied the long time limit (which in this setting is closely related to the massless electrons limit) of a single-species Vlasov-Poisson model with a Fokker-Planck collision term.
See also Herda \cite{Herda16} for a study of the massless electron limit in the case where an external magnetic field is also applied, leading to a fluid model for the electrons coupled with a kinetic model for ions.

In this paper, our focus will be on studying the VPME system \eqref{vpme}.
The nonlinearity in the Poisson equation is the key difference between ion and electron Vlasov-Poisson systems, and a source of additional mathematical richness.
Due to the difficulties created by this nonlinear coupling, the VPME system has been studied less widely than the electron model. However, global weak solutions were constructed by Bouchut \cite{Bouchut} in the three dimensional case.
In this paper we investigate the global well-posedness for the VPME system in dimension $d=2,3$.

More precisely, we begin by investigating the uniqueness of solutions and, in the spirit of the work of Loeper \cite{Loep} for the electron VP system, we show uniqueness for solutions of the VPME system with bounded density. 
This extension is very far from trivial since the exponential term $e^U$ in the Poisson equation creates several nonlinear effects.
The key estimates that allow us to achieve this result are obtained in Section \ref{sec:electric}.
More specifically, the first important step is to write the electric field $E$ as the sum of the electric field in VP, that we denote by $\bar E$, and a remainder $\widehat E$.
Then, in Proposition \ref{prop:regU}, we develop a series of regularity estimates on the electric fields $\bar E$ and $\widehat E$ that depend only on the $L^{(d+2)/d}$-norm of the density $\rho_f$. This is crucial since the latter norm on $\rho_f$ can be controlled uniformly (in time) thanks to our assumptions on the initial data (see Lemma \ref{lem:rho-Lp}).

With these estimates at our disposal, in Section \ref{sec:stab} we are able to perform a delicate Gronwall-type argument with respect to the Wasserstein distance in order to prove uniqueness and stability of solutions with bounded density. Central for this argument are the results from Proposition \ref{prop:Ustab}, showing quantitative stability estimates on the electric fields $\bar E$ and $\widehat E$ with respect to the density $\rho$.
This concludes the proof of uniqueness for solutions of the VPME system with bounded density. 

The remainder of the paper is devoted to finding sufficient conditions on the initial data that guarantee the global existence of solutions with bounded densities. As mentioned before, this kind of problem has been widely investigated in the setting of the Vlasov-Poisson system for electrons.
The key point here is that, thanks to the results obtained in Section \ref{sec:electric}, we are able to deal with the additional part $\widehat E$ of the electric field. 
In particular, in Sections \ref{sec:moments} and \ref{sec:construction} we prove propagation of moments of sufficiently high order to guarantee the existence of solutions with uniformly bounded density (see Theorem \ref{thm:existence}). It is worth mentioning that our approach is robust. Indeed, a similar strategy could be used to extend other known well-posedness results for the Vlasov-Poisson system to the massless electron case; for example, in a forthcoming paper we consider the case where the VPME system is posed on the whole space. Moreover, thanks to this well-posedness result new advances have been made regarding quasineutral and mean-field limits for the VPME system \cite{GPI20}.

\end{paragraph}

\section{Results and Outline}

\subsection{Main Result}

The main result of this paper is the global well-posedness of the VPME system \eqref{vpme}, posed on the torus $\TT^d,$ $d=2,3$, for large data with sufficiently rapid decay at infinity in the velocity variable. This is stated precisely in the following theorem.

\begin{thm} \label{thm:main}
Let $d = 2, 3$. Consider an initial datum $f_0 \in L^1 \cap L^\infty(\bt^d \times \br^d)$ satisfying
\be
f_0(x,v) \leq \frac{C_0}{1 + |v|^{k_0}},\quad\mbox{for some}\,\, k_0 > d, \qquad \int_{\TT^d \times \RR^d} |v|^{m_0} f_0(x,v) \di x \di v < + \infty,\quad\mbox{for some}\,\, m_0 > d(d-1) .
\ee
Then there exists a global-in-time weak solution $f \in C([0,\infty); \mc{P}(\bt^d \times \br^d))$ of the VPME system \eqref{vpme} with initial data $f_0$. This is the unique solution of \eqref{vpme} with initial datum $f_0$ such that 
$$
\rho_f \in L^\infty_{\text{loc}}([0,+\infty) ; L^\infty(\TT^d) ).
$$
In addition, if $f_0$ has compact support, then $f(t)$ has compact support for all times, with a bound on the size of the support which is locally uniform in time.
\end{thm}

\begin{remark}
The VPME system has an associated {\em energy} functional, which is formally a conserved quantity. It is defined by
\be \label{def:Ee}
\mc{E} [f ] := \frac{1}{2}\int_{\bt^d \times \br^d} |v|^2 f  \di x \di v + \frac{1 }{2} \int_{\bt^d} |\nabla U |^2 \di x +  \int_{\bt^d} U  e^{U } \di x ,
\ee
where $U $ is the solution of the nonlinear Poisson equation in \eqref{vpme}.
The weak solutions provided by Theorem~\ref{thm:main} conserve this energy: for all $t \geq 0$,
\be
\mc{E}[f(t)] = \mc{E}[f_0] .
\ee
\end{remark}

\begin{remark}
\label{rem:C1}
Notice that this theorem stipulates no regularity on the initial datum $f_0$, only that $f_0 \in L^1 \cap L^\infty(\bt^d \times \br^d)$. The resulting solutions thus will not in general be $C^1$ classical solutions. However, as we shall discuss later, since $\rho_f \in L^\infty_{\text{loc}}([0,+\infty) ; L^\infty(\TT^d) )$ they have a well-defined characteristic flow. Moreover the solution may be represented as the pushforward of the initial datum $f_0$ along this flow.

If the initial datum is additionally assumed to be $C^1$, 
then this characteristic flow can be used to show that the solution $f$ provided by Theorem~\ref{thm:main} is in fact $C^1$. Theorem~\ref{thm:main} thus also provides the global existence of $C^1$ classical solutions for $C^1$ initial data.
\end{remark}

\subsection{Strategy}

\subsubsection{Analysis of the Electric Field}

An important step of the proof is a toolbox of estimates on the electric field $E$, which we set out and prove in Section~\ref{sec:electric}. Following \cite{IHK1}, our starting observation is that we can consider $E$ as a sum of the electric field that appears in the electron model \eqref{vp}, plus a more regular nonlinear term.
That is, we decompose $E $ into the form $E = \bar E +\widehat E $, where
$$
\bar E =-\nabla \bar U ,\qquad \widehat E =-\nabla \widehat U ,
$$
and $\bar U $ and $\widehat U $ solve respectively
\be \label{electric-field-strategy}
 \Delta \bar U =1-\rho_f  ,\qquad  \Delta \widehat U =e^{\bar U +\widehat U } - 1.
\ee

In general, we would expect $\widehat E$ to be more regular than $\bar E.$ 
In Section~\ref{sec:electric} we prove that this intuition is rigorously true using techniques from calculus of variations to deal with the nonlinearity in the equation for $\widehat U$. 

In the context of analysing the VPME system \eqref{vpme}, what is crucial is to quantify this gain of regularity carefully, in particular in terms of its dependence on $\rho_f$. Specifically, we prove that if $\rho_f \in L^{(d+2)/d}(\bt^d)$, then $\widehat U \in C^{2,\alpha}(\bt^d)$ for some $\alpha > 0$, with a quantitative upper bound on the $C^{2,\alpha}$ norm that depends only on $\| \rho_f \|_{L^{(d+2)/d}}$.
The choice of $(d+2)/d$ as the integrability exponent is relevant because
this is a quantity that we expect to be bounded uniformly in time, as a consequence of the conservation of the energy functional $\mc{E}$ defined in \eqref{def:Ee} (see Subsection~\ref{sec:rho-interpolation}). 

These quantitative estimates explain in part the influence of the dimension $d$: the goal is to get estimates depending only on $L^p$ norms of $\rho_f$ up to order $p = (d+2)/d$. This exponent decreases as $d$ increases, while at the same time the gain of regularity provided by the ellipticity of the equation for $\widehat U$ also becomes weaker with increasing dimension. 

The proof of well-posedness is then divided into two auxiliary results. One is the uniqueness of solutions for VPME under the condition that the mass density $\rho_f$ is bounded in $L^\infty(\TT^d)$. The other is the global existence of solutions satisfying this condition under the assumptions of Theorem~\ref{thm:main}.

\subsubsection{Uniqueness}

In Section~\ref{sec:stab} we prove that uniqueness holds for the VPME system \eqref{vpme} under the assumption that the mass density $\rho_f$ is bounded in the sense of $L^\infty(\TT^d)$, locally uniformly in time. The same property has been known for the electron Vlasov-Poisson system \eqref{vp} ever since the work of Loeper \cite{Loep}. Our proof in the VPME setting makes use of Loeper's strategy to handle the electric field $\bar E$. However, the extension of this strategy to the VPME case requires nontrivial additional estimates for $\widehat E$ regarding its stability with respect to the inducing charge density $\rho_f$. These estimates are proved in Section~\ref{sec:electric}.

\subsubsection{Existence of Solutions}

In Sections \ref{sec:moments} and \ref{sec:construction}, we show that global-in-time solutions of the VPME system \eqref{vpme} exist for any initial datum $f_0 \in L^1 \cap L^\infty(\TT^d \times \RR^d)$ with a finite velocity moment of order $m_0 > d$ -- this is a wider range of data than that considered in Theorem~\ref{thm:main}, but there stronger assumptions are required for uniqueness.

The strategy of proof is based on showing that the VPME system propagates moments of sufficiently high order. This approach was used to prove well-posedness for the electron Vlasov-Poisson system, initially by Lions and Perthame \cite{Lions-Perthame} in the whole space case where $x \in \RR^3$. Pallard \cite{Pallard} (see also the previous contribution by Caglioti and Marchioro \cite{Caglioti-Marchioro}) then extended the range of moments that could be propagated in the whole space case and showed propagation of moments on the torus $\TT^3$. Chen and Chen \cite{Chen-Chen} adapted these techniques to push even further the range of moments that could be propagated.

In Section~\ref{sec:moments} we extend the proof in \cite{Chen-Chen} to the VPME case, proving an a priori estimate on the velocity moments of solutions of \eqref{vpme}.

Then, in Section~\ref{sec:construction}, we use this a priori estimate to prove global existence of solutions for the VPME system \eqref{vpme}. For this we first consider a regularised version of the VPME system. For the regularised system, solutions can be shown to exist using, for example, an adaptation of the methods of Dobrushin \cite{Dob}. Proving uniform moment estimates with respect to the regularization parameter, we then extract a limit point which we show is a solution of the original VPME system \eqref{vpme}. We explain the construction in detail in order to emphasise that no regularity is required on the initial datum $f_0$, and that the solutions so constructed are energy conserving.

\subsubsection{Remarks on Notation}

From now on we will use $\rho_{f}$ to denote the density generated by $f$.
Throughout the paper, we use the notation $C$ to denote an arbitrary positive constant, which may change from line to line. Subscripts are used to denote parameters upon which $C$ depends, for example $C_T$ denotes a constant depending in some way on another parameter $T$.

We identify the $d$-dimensional torus $\TT^d$ with the cube $\mc{Q}_d : = \left [- \frac{1}{2}, \frac{1}{2} \right ]^{d}$ or its translations, with appropriate identifications of the boundary. The distance between points on the torus is given by
\be \label{def:torus-metric}
d(x,y) = \inf_{\bb{Z}^d} |x - y  + k|.
\ee
With an abuse of notation, we will denote this distance by $|x-y|$ for points on the torus. Note that for all $x,y \in \TT^d$,
\be
|x-y| \leq \frac{1}{2} \sqrt{d} .
\ee

\section{Properties of the electric field} \label{sec:electric}

\subsection{Decomposition}

As explained above, we will split the electric field into a singular part, which behaves like the electric field in the Vlasov-Poisson system, and a more regular term. To be precise, we write $E $ in the form $\bar E +\widehat E $ where
$$
\bar E =-\nabla \bar U ,\qquad \widehat E =-\nabla \widehat U ,
$$
and $\bar U $ and $\widehat U $ solve respectively
\be \label{electric-field}
 \Delta \bar U =1-\rho  ,\qquad  \Delta \widehat U =e^{\bar U +\widehat U } - 1.
\ee
We will assume for convenience, and without loss of generality, that $\bar U $ has zero mean over the torus:
$$
\int_{\bt^d} \bar U  \di x = 0 .
$$
Notice that in this way $U :=\bar U +\widehat U $ solves
$$
 \Delta U =e^{U }-\rho .
$$
The remainder of this section is devoted to the study of the equations \eqref{electric-field}. We consider the existence and regularity of solutions as well as their stability with respect to the density $\rho $. We work under the assumption $\rho  \in L^\infty(\bt^d)$ (hence also in $L^{\frac{d+2}{d}}(\bt^d)$), since later we will work with solutions of \eqref{vpme} that have this degree of integrability.

\subsection{Regularity estimates on \texorpdfstring{$\bar U $ and $\widehat U $}{the potential}}

In this section we prove some a priori regularity estimates on the singular and regular parts of the potential $U  = \bar U  + \widehat U $. 
Our aim is to prove the following proposition.
\begin{prop}[Regularity estimates on $\bar U $ and $\widehat U $] \label{prop:regU} Let $d = 2,3$. Let $h \in L^\infty(\bt^d)$. Then there exist unique $\bar U   \in W^{1,2}(\bt^d)$ with zero mean and $\widehat U  \in W^{1,2}(\bt^d)$ satisfying
$$
 \Delta \bar U =1-h ,\qquad  \Delta \widehat U =e^{\bar U +\widehat U } - 1.
$$
Moreover we have the following estimates: for some constant $C_{\alpha,d}>0$,
\begin{equation}
\begin{array}{ll}
\lVert  \bar U  \rVert_{C^{0,\alpha}(\bt^d)}  \leq C_{\alpha, d} \,   \left (1 +  \lVert h \rVert_{L^{\frac{d+2}{d}}(\bt^d)} \right ), &\qquad
 \alpha \in \begin{cases} (0,1) \text{ if } d=2 \\ (0, \frac{1}{5}]  \text{ if } d=3 , \end{cases}  \\
\lVert \bar U  \rVert_{C^{1, \alpha}(\bt^d)}  \leq C_{\alpha, d} \,   \left (1 +  \lVert h \rVert_{L^{\infty}(\bt^d)} \right ),  &\qquad \alpha \in (0,1) ,\\
\|\widehat U  \|_{C^{1,\alpha}(\bt^d)}  \leq C_{\alpha,d}\, \exp{\Bigl(C_{\alpha,d} \,    \Bigl(1 +  \lVert h \rVert_{L^{\frac{d+2}{d}}(\bt^d)} \Bigr) \Bigr)}, &\qquad \alpha \in (0,1),\\
\|\widehat U  \|_{C^{2,\alpha}(\bt^d)}  \leq C_{\alpha,d}\,\exp\,\exp  {\Bigl(C_{\alpha,d}\,    \Bigl(1 +  \lVert h \rVert_{L^{\frac{d+2}{d}}(\bt^d)} \Bigr) \Bigr)}, & \qquad\alpha \in \begin{cases} (0,1) \text{ if } d=2 \\ (0, \frac{1}{5}]  \text{ if } d=3. \end{cases} 
\end{array}
\end{equation}

\end{prop}

\noindent The existence of a unique solution $\bar U  \in W^{1,2}(\bt^d)$ for $h \in L^2(\bt^d) \supset L^{\infty}(\bt^d)$ is well-known - see for example \cite[Chapter 6]{Evans}. In the following lemma, we recall some standard elliptic regularity estimates for this solution, that follow from Calder\'{o}n-Zygmund estimates for the Laplacian \cite[Section 9.4]{GT}, and Sobolev inequalities.

\begin{lem} \label{Ubar-reg}
Let $\bar U  \in \bar H$ satisfy
$$
  \Delta \bar U  = h .
$$
\begin{enumerate}[(i)]
\item If $h \in L^{\frac{d+2}{d}}(\bt^d)$, then for all $\alpha \in (0,1)$ (if $d=2$) or $\alpha \in (0, \frac{1}{5}]$ (if $d=3$) there exists a constant $C_{\alpha,d}>0$ such that
$$
\lVert  \bar U  \rVert_{C^{0,\alpha}(\bt^d)} \leq C_{\alpha, d}\, \left (1 +  \lVert h \rVert_{L^{\frac{d+2}{d}}(\bt^d)} \right ) .
$$
\item If $h \in L^{\infty}(\bt^d)$, then for any $\alpha \in (0,1)$, there exists a constant $C_{\alpha, d}$ such that
$$
\lVert \bar U  \rVert_{C^{1, \alpha}(\bt^d)} \leq C_{\alpha, d}\, \lVert h \rVert_{L^{\infty}(\bt^d)} .
$$
\end{enumerate}
\end{lem}

In order to prove estimates on the VPME equation \eqref{vpme}, we would ideally like to have good control of the regularity of the electric field. Unfortunately the estimates in Lemma~\ref{Ubar-reg} are not strong enough to provide Lipschitz regularity for $\nabla \bar U $ as we would like. However, a log-Lipschitz estimate is available. This well-known result is proved for instance in \cite[Lemma 8.1]{BM} for the case where the spatial domain is $\br^2$. For completeness we briefly recall the proof below for general $d$.

\begin{lem}[Log-Lipschitz regularity of $\bar E $] \label{lem:logLip}
Let $\bar U $ be a solution of 
$$
 \Delta \bar U = h
$$
for $h \in L^{\infty}(\bt^d)$. Then $\bar E : = - \nabla \bar U$ satisfies
\be \label{est:logLip}
| \bar E  (x) -  \bar E (y)| \leq    C_d \lVert h \rVert_{L^{\infty}} |x-y|  \left ( 1 + \log{\left ( \frac{\sqrt{d}}{2|x-y|} \right )} \right ) .
\ee
\end{lem}

To prove this lemma, we use the representation of $\bar E$ in terms of the Coulomb kernel $K = - \nabla G$, where $G$ satisfies \eqref{def:G}. We will need some information about the regularity properties of $K$. The following result shows that, near the singularity, the Coulomb kernel on the torus is comparable to the Coulomb kernel for the whole space. For a proof, see \cite{Titch} or \cite[Lemma 2.1]{IGP}.  

\begin{lem} \label{lem:G}
Let $G$ denote the Green's function for the negative Laplacian on the torus:
\be \label{def:G-2}
- \Delta G = \delta_0 - 1.
\ee
Then $G$ is a smooth function other than at zero: $G \in C^\infty(\TT^d \setminus \{0\})$. Moreover, on the ball $B_{1/4}(0)$ of radius $1/4$ and centred at zero, $G$ can be decomposed into the following form:
\be \label{def:G0}
G(x) = \begin{cases}
- \frac{1}{2 \pi} \log{|x|} + G_0(x) & d=2 \\
 \frac{1}{ |\bb{S}_{d-1}| |x|^{d-2}} + G_0(x) & d\ge3 ,
\end{cases}
\ee
where $G_0 \in C^\infty(\overline{B_{1/4}(0)})$ is a smooth function, and $|\bb{S}_{d-1}|$ denotes the surface area of the unit sphere in dimension $d$.
\end{lem}

\begin{proof}[Proof of Lemma~\ref{lem:logLip}]
We use the representation of $\bar E$ using the Coulomb kernel: observe that
\be
\bar E = K \ast h
\ee
where $K = - \nabla G$ is the Coulomb kernel on the torus, with $G$ defined by \eqref{def:G-2}. 
Thus
\be \label{barE-convolution}
|\bar E(x) - \bar E(y)| = \left | \int_{\bt^d} \left [ K(x-z) - K(y-z) \right ] h(z) \di z \right | .
\ee

By Lemma~\ref{lem:G}, on the ball $B_{1/4}(0)$ the kernel $K$ has the representation
\be \label{K:decomp}
K(x) = - C_d \frac{x}{|x|^d} + K_0(x)
\ee
for some $K_0 \in C^1(\bt^d)$. In particular note that $K \in L^1(\TT^d)$ and so
\be
|\bar E(x) - \bar E(y)| \leq 2 \| K \|_{L^1(\TT^d)} \| h \|_{L^\infty(\TT^d)}.
\ee
Therefore, it suffices to prove the estimate \eqref{est:logLip} for small values of $|x-y|$.

We evaluate the integral \eqref{barE-convolution} by identifying the torus $\TT^d$ with the cube $x + \mc{Q}_d$. We then divide the cube into a region close to the singularity of $K(x-z)$ and region far from the singularity. Define the regions
\begin{align}
A_1  = \{ z \in x + \mc{Q}_d: |x-z| \leq 2 |x-y| \},\qquad
A_2  = \{ z \in x + \mc{Q}_d : 2 |x-y| \leq |x-z| \} .
\end{align}
Then let
\be
I_i : = \int_{z \in A_i}  \left [ K(x-z) -  K(y-z)  \right ] h(z) \di z .
\ee 

We now assume that $|x-y| \leq \frac{1}{12}$. This is chosen so that for all $z \in A_1$,
\be
|x-z| \leq 2|x-y| \leq \frac{1}{4}, \quad |y-z| \leq 3 |x-y| \leq \frac{1}{4}.
\ee
Then $I_1$ can be bounded in the following way:
\be
I_1 \leq C_d \left \lvert \int_{z \in A_1}  \left [ \frac{x-z}{|x-z|^d} -  \frac{y-z}{|y-z|^d}  \right ] h(z) \di z \right \rvert + \int_{z \in A_1} |K_0(x-z) - K_0(y-z)|  |h(z)|\di z .
\ee
For the second term, we note that, for all $t \in [0,1]$,
\be
|x-z + t(y-x)| \leq (2+t)|x-y| \leq \frac{1}{4},
\ee
that is, the line segment $(1-t)x + ty - z$ is contained in the ball $B_{1/4}(0)$, on which $K_0$ is a $C^1$ function. Thus
\be
|K_0(x-z) - K_0(y-z)| \leq \| \nabla K_0\|_{L^\infty(B_{1/4}(0))} |x-y| .
\ee
The second term is therefore bounded by
\be
\int_{z \in A_1} |K_0(x-z) - K_0(y-z)|  |h(z)|\di z \leq C_K \| h \|_{L^\infty(\TT^d)} |x-y| .
\ee
We bound the first term by integrating over the singularity:
\begin{align}
\left \lvert \int_{z \in A_1}  \left [ \frac{x-z}{|x-z|^d} -  \frac{y-z}{|y-z|^d}  \right ] h(z) \di z \right \rvert & \leq \| h \|_{L^{\infty}(\bt^d)} \left ( \int_{z \in A_1} |x-z|^{-(d-1)} \di z + \int_{z \in A_1} |y-z|^{-(d-1)} \di z \right ) \\
& \leq  \| h \|_{L^{\infty}(\bt^d)} \left ( \int_{|u| \leq 2 |x-y|} |u|^{-(d-1)} \di u + \int_{|u| \leq 3 |x-y|} |u|^{-(d-1)} \di u \right ) \\
& \leq C \| h \|_{L^{\infty}(\bt^d)}  |x-y| .
\end{align}

On $A_2$, we use the derivative of $K$. By Lemma~\ref{lem:G}, we have the estimate
\be
|\nabla K (x) | \leq \begin{cases} C_d |x|^{-d} + \| \nabla K_0 \|_{L^\infty(B_{1/4}(0))} & x \in B_{1/4}(0) \\
\| \nabla K \|_{L^\infty(B_{1/4}^c(0))} & x \notin B_{1/4}(0) .
\end{cases}
\ee
Consider the straight line segment $[(1-t)x + ty - z]_{t \in [0,1]}$ connecting the points $x-z$ and $y-z$. Observe that, since $|x-z| > 2 |x-y|$, on this line segment
\be
|(1-t)x + ty - z| \geq |x-z| - t|x-y| \geq (1 - \frac{t}{2}) |x-z|  \geq \frac{1}{2} |x-z|.
\ee
Thus the derivative can be bounded by
\be
|\nabla K (x)| \leq C_d \left ( 1 + |x-z|^{-d} \right ) .
\ee
It follows that
\be
|K(x-z) - K(y-z)| \leq C_d \left ( 1 + |x-z|^{-d} \right ) |x-y| .
\ee
Therefore
\begin{align}
I_2 & \leq C_d \| h \|_{L^\infty(\TT^d)} |x-y| \int_{z \in A_2} \left ( 1 + |x-z|^{-d} \right ) \di z \\
& \leq  C_d \| h \|_{L^\infty(\TT^d)} |x-y| \left ( 1 + \int_{\frac{\sqrt{d}}{2} \geq |u| \geq 2|x-y|} |u|^{-d}  \di z \right ) \\
& \leq C_d \| h \|_{L^\infty(\TT^d)} |x-y| \left ( 1 - \log{\frac{2 |x-y|}{\sqrt{d}}} \right )  .
\end{align}

Altogether we have proved that
\begin{align}
\left | \bar E  (x) - \bar E (y) \right | \leq C_d    \| h \|_{L^\infty(\bt^d)} |x-y| \left ( 1 - \log{\frac{2 |x-y|}{\sqrt{d}}} \right ) ,
\end{align}
which concludes the proof.
\end{proof}

\subsection{Existence and regularity of \texorpdfstring{$\widehat U $}{U}}

In this section we will prove the existence of $\widehat U $ and some useful regularity estimates. We note that the proposition below holds in any dimension $d$.

\begin{prop}[Existence and H\"older regularity of $\widehat U $]
\label{prop:hatU-e}
Assume that 
\be \label{barU-apr}
\lVert \nabla \bar U  \rVert_{L^2(\bt^d)} + \|\bar U \|_{L^{\infty}(\mathbb T^d)}\leq M_1.
\ee
Then the equation
\be \label{hatU-eqn}
  \Delta\widehat U =e^{(\bar U +\widehat U )} - 1 \qquad \text{on }\mathbb T^d 
\ee
has a unique solution in $W^{1,2}(\bt^d)$. Furthermore, for any $\alpha \in (0,1)$ this solution satisfies
$$
\|\widehat U \|_{C^{1,\alpha}(\mathbb T^d)} \leq C    \(1 + e^{2M_1}\) .
$$
If in addition, for some $\alpha \in (0,1)$, $\bar U \in C^{0,\alpha}(\bt^d)$, with
$$
\lVert \bar U  \rVert_{C^{0,\alpha}(\bt^d)} \leq M_2 ,
$$
then $\widehat U  \in C^{2,\alpha}(\bt^d)$ with
$$
\lVert \widehat U  \rVert_{ C^{2,\alpha}(\bt^d)} \leq C \exp{\left [ C \left ( M_1 +    (1 + e^{2M_1}) \right) \right ]} \( M_2 +    (1 + e^{2M_1}) \) .
$$
\end{prop}

\begin{proof}
We prove existence of $\widehat U $ by finding a minimiser for the functional
$$
h \mapsto E[h]:= \int_{\mathbb T^d} \frac12 |\nabla h|^2 + \Bigl(e^{\bar U +h} - h\Bigr) \di x
$$
among all periodic functions $h \in W^{1,2}(\bt^d)$.
Indeed \eqref{hatU-eqn} is the Euler-Lagrange equation of the above functional.

Notice that since $E[h]$ is a strictly convex functional, solutions of the Euler-Lagrange equation
are minimisers and the minimiser is unique.
Let us now prove existence of a minimiser using the direct method of Calculus of Variations.

Consider a minimising sequence $h_k$, that is
$$
E[h_k] \to \inf_{h}E[h]=:\alpha.
$$
We then need to prove that $h_k$ is uniformly bounded in $W^{1,2}(\bt^d)$ and that the functional $E[h]$ is lower semicontinuous.

Observe that, by choosing $h=-\bar U $, we get 
\be
\alpha \leq E[-\bar U ] =\int_{\mathbb T^d} \frac12 |\nabla \bar U |^2 +  \left ( 1 + \bar U  \right ) \di x
\ee
Using the $L^\infty \cap W^{1,2}(\bt^d)$ bound  \eqref{barU-apr} for $\bar U $, we deduce that
\be
\alpha \leq \frac{1}{2}M_1^2+   (1+M_1) .
\ee
Thus, for sufficiently large $k$,
\be\label{eq:C_1}
E[h_k] \leq C  (1+M_1^2)=: C_1 .
\ee
We observe that
$$
e^{\bar U +s}-s= e^{\bar U } e^s-s\ge e^{-M_1}e^s-s\ge |s|-C_2,
$$
thus, 
\be\label{eq:equibound}
\int_{\mathbb T^d} e^{\bar U +h_k}-h_k \di x \ge \int_{\mathbb T^d} |h_k|-C_2 \di x.
\ee
By equation \eqref{eq:C_1} and \eqref{eq:equibound},
$$
\int_{\mathbb T^d} \frac{1}{2}|\nabla h_k|^2+   (|h_k|-C_2 ) \di x \le E[h_k] \le\alpha+1 \leq C_1+1.
$$
Therefore, by Poincar\'e inequality we obtain that $h_k$ are equibounded in $W^{1,2}(\bt^d):$
$$
||\nabla h_k||_{L^2(\bt^d)}+||h_k||_{L^2(\bt^d)}\le C_3.
$$
Hence, by weak compactness of $W^{1,2}(\bt^d)$, up to a subsequence $h_k$ converges weakly in $W^{1,2}(\bt^d)$ to a function $\widehat U $:
$$
h_k {\rightharpoonup} \widehat U \qquad \text{in ${W^{1,2}(\bt^d)}$}.
$$
Since $W^{1,2}(\bt^d)$ is compactly embedded in $L^2(\bt^d)$, we also have strong convergence:
$$
h_k {\rightarrow} \widehat U \qquad \text{in ${L^{2}(\bt^d)}$}.
$$
Then, up to a further subsequence, we have 
$$
h_{k} \rightarrow \widehat U  \quad a.e.
$$
We claim that $\widehat U $ is a minimiser.
Indeed, by the weak convergence in $W^{1,2}(\bt^d)$ and by strong convergence in $L^2(\bt^d)$ (which implies strong convergence in $L^1(\bt^d)$), and by the lower semicontinuity of the norm under weak convergence we have that:
$$
\liminf_{k\rightarrow \infty}\int_{\mathbb T^d}\frac{1}{2} |\nabla h_k|^2 \di x \ge\int_{\mathbb T^d}\frac{1}{2} |\nabla \widehat U |^2 \di x,
$$
$$
\lim_{k\rightarrow \infty}-\int_{\mathbb T^d} h_k \di x =-\int_{\mathbb T^d} \widehat U  \di x.
$$
Also, by Fatou's Lemma,
$$
\liminf_{k\rightarrow \infty}\int_{\mathbb T^d} e^{\bar U +h_k} \di x \ge \int_{\mathbb T^d} \liminf_{k\rightarrow \infty} e^{\bar U +h_k} \di x =\int_{\mathbb T^d}e^{\bar U +\widehat U } \di x.
$$
In conclusion we obtained that
$$
\alpha=\lim_{k\rightarrow \infty}E[h_k]\ge E[\widehat U ],
$$
which proves that $\widehat U $ is a minimiser.
We now need to check that $\widehat U $ solves the Euler-Lagrange equations.
First, observe that
$$
C_1\ge \alpha=E[\widehat U ]=\int_{\mathbb T^d} \frac{1}{2}|\nabla \widehat U |^2+    \( e^{\bar U }e^{\widehat U }-\widehat U  \) \di x,
$$
and therefore $\nabla \widehat U \in L^2(\bt^d)$, $\widehat U \in L^1(\bt^d)$, $e^{\bar U }\in L^\infty(\bt^d)$, and $e^{\widehat U }\in L^1(\bt^d).$
Let $\phi \in C^\infty(\bt^d), \eta>0$. By minimality of $\widehat U $, 
$$
E[\widehat U ]\le E[\widehat U +\eta \phi].
$$
Then,
\begin{align}
0&\le \frac{E[\widehat U +\eta \phi]-E[\widehat U ]}{\eta}=\frac{1}{\eta}\bigg(\int_{\mathbb T^d} \frac{1}{2}|\nabla \widehat U +\eta\nabla\phi|^2-\frac{1}{2}|\nabla \widehat U |^2 \di x \bigg) \\
&+\frac{1}{\eta}    \bigg(\int_{\mathbb T^d} e^{\bar U }e^{\widehat U +\eta\phi}-e^{\bar U }e^{\widehat U }  \di x \bigg) +\frac{1}{\eta}    \bigg(\int_{\mathbb T^d} -(\widehat U +\eta\phi)+\widehat U   \di x \bigg)\\
&=\int_{\mathbb T^d}\nabla \widehat U \cdot \nabla \phi+\frac{\eta |\nabla \phi|^2}{2}  \di x +   \int_{\mathbb T^d} e^{\bar U  + \widehat U } \,\frac{e^{\eta\phi}-1}{\eta} \di x -   \int_{\mathbb T^d}\phi  \di x.
\end{align}
In the limit as $\eta$ goes to $0$ we obtain
\begin{align}
0\le \lim_{\eta\rightarrow 0}\frac{E[\widehat U +\eta \phi]-E[\widehat U ]}{\eta}=\int_{\mathbb T^d}\nabla \widehat U \cdot \nabla \phi \di x +  \int_{\mathbb T^d} e^{\bar U +\widehat U }\phi \di x -  \int_{\mathbb T^d}\phi \di x \quad \mbox{for\, all\,} \phi \in C^\infty(\bt^d).
\end{align}
Since the latter inequality is valid both for $\phi$ and for $-\phi,$ we have that
\begin{align}\label{eq:eulero_lagrange}
0=\int_{\mathbb T^d}\nabla \widehat U \cdot \nabla \phi+  \(e^{\bar U + \widehat U }-1\)\phi \di x=\int_{\mathbb T^d} -\Delta \widehat U \, \phi+  \(e^{\bar U +\widehat U }-1\)\phi \di x  \quad \mbox{for\, all\,} \phi \in C^\infty(\bt^d).
\end{align}
By the arbitrariness of $\phi$, \eqref{eq:eulero_lagrange} implies that
\be\label{eq:EL}
  \Delta \widehat U =e^{\bar U +\widehat U }-1 \qquad \text{on }\mathbb T^d.
\ee
We now prove the desired estimates on $\widehat U $. Our goal is to control $\|\widehat U \|_{C^{2,\alpha}(\mathbb T^d)}$ .
To do that, it is enough to prove that
\be \label{U-Holder-ap}
\|\widehat U \|_{C^{0,\alpha}(\mathbb T^d)} \leq C .
\ee
Indeed, since $\bar U \in C^{0,\alpha},$ then by equation \eqref{eq:EL} we will have $\Delta \widehat U \in C^{0,\alpha}$ and, thanks to Schauder's estimates \cite[Chapter 4]{GT}, this implies that $\widehat U \in C^{2,\alpha}$. To obtain \eqref{U-Holder-ap}, we will use a priori estimates on the equation satisfied by $\widehat U$. For this we will need suitable $L^p(\bt^d)$ estimates on $e^U$, which we will derive via energy estimates, that is, by using appropriate test functions in \eqref{eq:eulero_lagrange}.

In order to give a meaning to equation \eqref{eq:eulero_lagrange}, we need $\phi$ to be at least in $L^\infty\cap W^{1,2}(\bt^d)$. We will now build a test function in $L^\infty\cap W^{1,2}(\bt^d)$ that will allow us to prove a regularity estimate on $\widehat U $. 
Let us consider the truncated function
$$
\widehat U_{k}:=(\widehat U \wedge k), \quad \text{for all} \, k\in \mathbb N.
$$
Since $e^{\widehat U_{k}}\in L^\infty(\bt^d)$ and $\nabla \widehat U  \in L^2(\bt^d)$,
$$
\nabla e^{\widehat U_{k}}=e^{\widehat U_{k}}\nabla \widehat U_{k}=e^{\widehat U_{k}}\nabla \widehat U \chi_{\{\widehat U <k\}} \in L^2(\bt^d);
$$
thus $e^{\widehat U_{k}}\in L^\infty \cap W^{1,2}(\bt^d),$ and we can use it as a test function in equation \eqref{eq:eulero_lagrange}:
\begin{align}
0&=  \int_{\mathbb T^d}\nabla \widehat U \cdot \nabla e^{\widehat U_{k}} \di x + \int_{\bt^d} \(e^{\bar U }e^{(\widehat U +\widehat U_{k})}-e^{\widehat U_{k}}\) \di x \\
& = \int_{\mathbb T^d}\nabla \widehat U \cdot e^{\widehat U_{k}}\nabla \widehat U \chi_{\{\widehat U <k\}} \di x + \int_{\bt^d} \(e^{\bar U }e^{(\widehat U +\widehat U_{k})}-e^{\widehat U_{k}}\) \di x \\
&= \int_{\mathbb T^d}|\nabla \widehat U |^2 e^{\widehat U_{k}} \chi_{\{\widehat U <k\}} \di x +\int_{\mathbb T^d}e^{\bar U }e^{(\widehat U +\widehat U_{k})} \di x -\int_{\mathbb T^d}e^{\widehat U_{k}} \di x . \label{eU-L1}
\end{align}
Since $\int_{\mathbb T^d}|\nabla \widehat U |^2 e^{\widehat U_{k}} \chi_{\{\widehat U <k\}} \di x \ge 0,$ and $C_0:=e^{-M_1}\le e^{\bar U }\le e^{M_1}$, \eqref{eU-L1} implies that
\be
C_0 \int_{\mathbb T^d}e^{\widehat U +\widehat U_{k}}\le \int_{\mathbb T^d}e^{\widehat U_{k}}.
\ee
By definition of $\widehat U_{k}$ we have that $e^{\widehat U_{k}}$ is increasing and converges monotonically to $e^{\widehat U },$ hence by the Monotone Convergence Theorem
\be
C_0 \int_{\mathbb T^d}e^{\widehat U +\widehat U }=C_0 \int_{\mathbb T^d}e^{2\widehat U }\le \int_{\mathbb T^d}e^{\widehat U } ,
\ee
and we obtain that if $e^{\widehat U }\in L^1(\bt^d),$ then $e^{\widehat U }\in L^2(\bt^d).$ In particular, since $C_0=e^{-M_1}$,
\begin{align} \label{bar-exp_l2-l1}
\lVert e^{\widehat U } \rVert_{L^2(\bt^d)} & \leq e^{M_1/2} \left( \int_{\bt^d} e^{\widehat U } \right)^{1/2}.
\end{align}
Since $\widehat U $ is a solution of \eqref{eq:EL}, we have
$$
0 =   \int_{\bt^d} \Delta \widehat U \di x = \int_{\bt^d} e^{U }- 1 \di x .
$$
Since $U = \bar U + \widehat U $, it follows that
$$
1 = \int_{\bt^d} e^{\bar U + \widehat U } \di x \geq e^{-M_1} \int_{\bt^d} e^{\widehat U } \di x .
$$
Thus
\be
\lVert e^{\widehat U } \rVert_{L^1(\bt^d)} = \int_{\bt^d} e^{\widehat U } \di x \leq e^{M_1},
\ee
and hence \eqref{bar-exp_l2-l1} implies that
\begin{align} \label{bar-exp_l2}
\lVert e^{\widehat U } \rVert_{L^2(\bt^d)} & \leq e^{M_1} .
\end{align}
If we now use the function $e^{2\widehat U_{k}}$ as test function in the equation \eqref{eq:eulero_lagrange}, we obtain
\begin{align}
0&=   \int_{\mathbb T^d}\nabla \widehat U \cdot \nabla e^{2\widehat U_{k}} \di x + \int_{\bt^d} \(e^{\bar U }e^{\widehat U+2\widehat U_{k}} e^{2\widehat U_{k}}\) \di x \\
& =   \int_{\mathbb T^d}2\nabla \widehat U \cdot e^{2\widehat U_{k}}\nabla \widehat U \chi_{\{\widehat U <k\}} \di x + \int_{\bt^d}\(e^{\bar U }e^{\widehat U +2\widehat U_{k}}-e^{2\widehat U_{k}}\) \di x\\
&=2   \int_{\mathbb T^d}|\nabla \widehat U |^2 e^{2\widehat U_{k}} \chi_{\{\widehat U <k\}} \di x +\int_{\mathbb T^d}e^{\bar U }e^{\widehat U +2\widehat U_{k}} \di x -\int_{\mathbb T^d}e^{2\widehat U_{k}} \di x .
\end{align}
Thus, as in the previous case,
\be
C_0 \int_{\mathbb T^d}e^{\widehat U +2\widehat U_{k}} \di x \le \int_{\mathbb T^d}e^{2\widehat U_{k}} \di x,
\ee
and by Monotone Convergence as $k\to \infty$, recalling \eqref{bar-exp_l2} we get
\be
C_0 \int_{\mathbb T^d}e^{3\widehat U }\le \int_{\mathbb T^d}e^{2\widehat U }\le e^{2M_1}.
\ee
Hence
\be
\lVert e^{\widehat U } \rVert_{L^3(\bt^d)}^3 \leq e^{3M_1} .
\ee
Iterating $n$ times, with $n>d,$ we obtain
$$
\lVert e^{\widehat U } \rVert_{L^n(\bt^d)} \leq e^{M_1}
$$
and hence
$$
  \Delta \widehat U =e^{\bar U +\widehat U }-1 \in L^n(\bt^d),
$$
with
\be \label{eU-Ln}
\lVert e^{\bar U +\widehat U }-1 \rVert_{L^n(\bt^d)} \leq 1 + e^{2M_1} .
\ee
By standard regularity estimates for the Poisson equation \cite[Section 9.4]{GT},
$$
\lVert \widehat U \rVert_{ W^{2,n}(\bt^d)} \leq C    \(1 + e^{2M_1}\) .
$$
Using Sobolev embedding for $n$ sufficiently large, we deduce that for any $\alpha \in (0,1)$, $\widehat U \in C^{1, \alpha}(\bt^d)$, with
$$
\lVert \widehat U \rVert_{C^{1, \alpha}(\bt^d)} \leq C    \(1 + e^{2M_1}\) .
$$
Then, if
$$
\lVert \bar U \rVert_{C^{0,\alpha}(\bt^d)} \leq M_2,
$$
we have
$$
\lVert U \rVert_{C^{0,\alpha}(\bt^d)} \leq M_2 + C    \(1 + e^{2M_1}\),
$$
and so
$$
\lVert e^{U }\rVert_{C^{0,\alpha}(\bt^d)} \leq C \exp{\left [ C \left ( M_1 +    \(1 + e^{2M_1}\) \right) \right ]} \( M_2 +    \(1 + e^{2M_1}\) \) .
$$
Thus by Schauder estimates \cite[Chapter 4]{GT}
\begin{align}
\lVert \widehat U \rVert_{C^{2,\alpha}(\bt^d)} & \leq C \left ( \lVert \widehat U \rVert_{L^{\infty}(\bt^d)} +    \lVert e^{U }- 1 \rVert_{C^{0, \alpha}(\bt^d)} \right ) \\
& \leq C\,    \exp{\left [ C \left ( M_1 +    \(1 + e^{2M_1}\) \right) \right ]}  \( M_2 +    \(1 + e^{M_1}\) \) .
\end{align}

\end{proof}

\subsection{Stability with respect to the charge density} \label{sec:electric-stability}
Next we wish to study the stability of the electric field $\nabla U $ with respect to the charge density $\rho $. To measure stability we shall use the $2$-Wasserstein distance.
\subsubsection{Wasserstein Distances}\label{Wasserstein Distances}

To define the Wasserstein (or Monge-Kantorovich) distances, we first define a \textit{coupling} between two measures. Let $(\Omega, \mc{F})$ be a measurable space, and let $\mu, \nu \in \mc{P}(\Omega)$ be probability measures. A coupling is a measure on the product space, $\pi \in \mc{P}(\Omega \times \Omega)$, which has marginals $\mu$ and $\nu$. This means that for all $A \in \mc{F}$,
\be
\pi(A \times \Omega) = \mu(A) , \quad
\pi(\Omega \times A) = \nu(A) .
\ee
We denote the set of couplings of $\mu$ and $\nu$ by $\Pi(\mu, \nu)$. We now give the definition of the Wasserstein distances. For further background on optimal transport distances, see \cite{Vil03}.

\begin{defi}[Wasserstein distances]  \label{def:Wass}
Let $p \in [1, \infty)$. Let $(\Omega, d)$ be a Polish space and let $\mc{F}$ be its Borel $\sigma$-algebra. Let $\mu, \nu$ be Borel measures satisfying
\be \label{measure-moment}
\int_\Omega d(x, x_0)^p \mu(\di x), \quad \int_\Omega d(x, x_0)^p \nu(\di x) < \infty
\ee
for some $x_0$. The Wasserstein distance of order $p$ between $\mu$ and $\nu$ is defined by
\be
W_p^p(\mu, \nu) : = \inf_{\pi \in \Pi(\mu,\nu)} \int_{\Omega \times \Omega} d(x,y)^p \pi(\di x \di y) .
\ee
\end{defi}

In this paper, we will use the case $\Omega = \bt^d \times \br^d$. The distance $d$ is given by
\be
d\left ((x_1, v_1), (x_2, v_2) \right ) = |x_1 - x_2| + |v_1 - v_2|,
\ee
where for the $x$ coordinate we use the distance on the torus, defined by \eqref{def:torus-metric}.

 The main result of Section \ref{sec:electric-stability} is the following proposition. Again this holds in every dimension.
\begin{prop} \label{prop:Ustab}
For each $i=1,2$, let $\bar U _i$ be a solution of
\be 
 \Delta \bar U _i= h_i - 1,
\ee
where $h_i \in L^{\infty} \cap L^{(d+2)/d}(\bt^d)$. Then
\be \label{stab-Ubar}
\lVert \nabla \bar U _1 - \nabla \bar U _2 \rVert^2_{L^2(\bt^d)} \leq   \max_i\, \lVert h_i \rVert_{L^{\infty}(\bt^d)}  \, W^2_2(h_1, h_2) .
\ee
In addition, let $\widehat U _i$ be a solution of
\be 
 \Delta \widehat U _i= e^{\bar U _i + \widehat U _i} - 1 .
\ee
Then
\begin{align}
\label{stab-Uhat}
\lVert \nabla \widehat U _1 - \nabla \widehat U _2 \rVert^2_{L^2(\bt^d)} &\leq \exp\,\exp  {\left [ C_d    \left ( 1 + \max_i\, \lVert h_i \rVert_{L^{(d+2)/d}(\bt^d)}  \right ) \right ]}
 \max_i\, \lVert h_i \rVert_{L^{\infty}(\bt^d)} \, W^2_2(h_1, h_2).
\end{align}
\end{prop}

For the Poisson part, we will use a stability estimate, with respect to the Wasserstein distance, for Poisson's equation on the torus. A proof may be found in \cite{IHK2}; see also \cite{Loep} for the case where $x \in \br^d$.
\begin{lem} \label{lem:Loep}
For each $i=1,2$, let $\bar U _i$ be a solution of
\be 
 \Delta \bar U _i= h_i - 1,
\ee
where $h_i \in L^{\infty}(\bt^d)$. Then
$$
\lVert \nabla \bar U _1 - \nabla \bar U _2 \rVert^2_{L^2(\bt^d)} \leq   \max_i\, \lVert h_i \rVert_{L^{\infty}(\bt^d)} \, W^2_2(h_1, h_2) .
$$
\end{lem}

For the nonlinear part we derive a suitable energy estimate.
\begin{lem} \label{lem:hatU-stab}
For each $i=1,2$, let $\widehat U _i \in  W^{1,2} \cap L^{\infty}(\bt^d)$ be a solution of
\be \label{PoiU}
 \Delta \widehat U_i =e^{\bar U_i +\widehat U_i } - 1 ,
\ee
for some given potentials $\bar{U}_i  \in L^{\infty}(\bt^d)$ . Then
\be  \label{H1dot}
  \lVert \nabla \widehat{U}_1  - \nabla \widehat{U}_2  \rVert^2_{L^2(\bt^d)} \leq \widehat{C}  \lVert \bar{U}_1  - \bar{U}_2  \rVert^2_{L^2(\bt^d)} ,
\ee
where $\widehat{C} $ depends on the $L^{\infty}$ norms of $\widehat{U}_i $ and $\bar{U}_i $. More precisely, $\widehat{C} $ can be chosen such that
$$
\widehat{C}   \leq  \exp{\left [ C_d \left ( \max_i\,  \lVert \bar{U}_i  \rVert_{L^{\infty}(\bt^d)} + \max_i\, \lVert \widehat{U}_i  \rVert_{L^{\infty}(\bt^d)}  \right )\right ]} ,
$$
for some sufficiently large dimensional constant $C_d$.

\end{lem}

\begin{proof}
For convenience, we define the constant 
$$
A  : = \exp{\left [  \max_i\,  \lVert \bar{U}_i  \rVert_{L^{\infty}(\bt^d)} + \max_i\, \lVert \widehat{U}_i  \rVert_{L^{\infty}(\bt^d)}\right ]}
$$
which will be fixed throughout the proof.

Subtracting the two equations \eqref{PoiU}, we deduce that $ \widehat{U}_1  - \widehat{U}_2 $ satisfies
\be  \label{Poi-diff}
  \Delta ( \widehat{U}_1  - \widehat{U}_2  ) = e^{\bar{U}_1 +\widehat{U}_1 } - e^{\bar{U}_2 +\widehat{U}_2 }  = e^{\bar{U}_1 } \left ( e^{\widehat{U}_1 } - e^{\widehat{U}_2 } \right ) + e^{\widehat{U}_2 } \left ( e^{\bar{U}_1 } - e^{\bar{U}_2 } \right ) .
\ee

The weak form of \eqref{Poi-diff} extends by density to test functions in $L^\infty \cap W^{1,2}(\bt^d)$. Since $ \widehat{U}_1  - \widehat{U}_2 $ has this regularity by assumption, it is an admissible test function. Hence
\begin{align} \label{H1dot-IBP}
-   \int_{\bt^d} \lvert \nabla \widehat{U}_1  - \nabla \widehat{U}_2  \rvert^2 \di x &=\int_{\bt^d } e^{\bar{U}_1 } \left ( e^{\widehat{U}_1 } - e^{\widehat{U}_2 } \right ) ( \widehat{U}_1  - \widehat{U}_2  )  \di x \\
& \qquad +  \int_{\bt^d } e^{\widehat{U}_2 } \left ( e^{\bar{U}_1 } - e^{\bar{U}_2 } \right ) ( \widehat{U}_1  - \widehat{U}_2  )  \di x   = : I_1 + I_2 .
\end{align}

Observe that $(e^x - e^y)(x-y)$ is always non-negative. Furthermore, by the Mean Value Theorem applied to the function $x \mapsto e^x,$ we have a lower bound
\be
(e^x - e^y)(x-y) \geq e^{\min\{x,y\}} (x-y)^2 .
\ee
We use this to bound $I_1$ from below:
\be  \label{H1dot-term1}
I_1 \geq e^{- \lVert \bar U_1  \rVert_{L^{\infty}(\bt^d)} - \max_i\, \lVert \widehat{U}_i  \rVert_{L^{\infty}(\bt^d)}} \lVert \widehat{U}_1  - \widehat{U}_2  \rVert_{L^2(\bt^d)}^2  \geq A ^{-1}  \lVert \widehat{U}_1  - \widehat{U}_2  \rVert_{L^2(\bt^d)}^2 .
\ee
For $I_2$ we use the fact that, again by the Mean Value Theorem,
\be
|e^x - e^y| \leq e^{\max\{x,y\}} |x-y|.
\ee
Therefore
\be 
I_2  \leq e^{\lVert \widehat{U}_2  \rVert_{L^{\infty}(\bt^d)} + \max_i\, \lVert \bar U_i  \rVert_{L^{\infty}(\bt^d)}} \int_{\bt^d} |\bar{U}_1  - \bar{U}_2 | |\widehat{U}_1  - \widehat{U}_2 |  \di x  \leq A  \int_{\bt^d} |\bar{U}_1  - \bar{U}_2 | |\widehat{U}_1  - \widehat{U}_2 |  \di x .
\ee
By the Cauchy-Schwarz inequality, for any choice of $\alpha > 0$ 
\be  \label{H1dot-term2}
I_2 \leq A  \left ( \alpha \lVert \bar{U}_1  - \bar{U}_2  \rVert_{L^2(\bt^d)}^2 + \frac{1}{4\alpha} \lVert \widehat{U}_1  - \widehat{U}_2  \rVert_{L^2(\bt^d)}^2   \right ) .
\ee
Substituting \eqref{H1dot-term1} and \eqref{H1dot-term2} into \eqref{H1dot-IBP}, we obtain
\begin{align} \label{H1dot-alpha}
  \int_{\bt^d} \lvert \nabla \widehat{U}_1-\nabla \widehat{U}_2  \rvert^2 \di x & \leq A  \left ( \alpha \lVert \bar{U}_1  - \bar{U}_2  \rVert_{L^2(\bt^d)}^2 + \frac{1}{4\alpha} \lVert \widehat{U}_1-\widehat{U}_2  \rVert_{L^2(\bt^d)}^2   \right ) - A ^{-1} \lVert \widehat{U}_1  - \widehat{U}_2  \rVert_{L^2(\bt^d)}^2 .
\end{align}
We wish to choose $\alpha$ as small as possible such that
\be
\frac{A }{4 \alpha} - A ^{-1} \leq 0 .
\ee
Thus the optimal choice is $\alpha = \frac{A ^2}{4}$. Substituting this into \eqref{H1dot-alpha} gives
\be 
  \int_{\bt^d} \lvert \nabla \widehat{U}_1  - \nabla \widehat{U}_2  \rvert^2 \di x \leq \frac{1}{4} A ^3 \lVert \bar{U}_1  - \bar{U}_2  \rVert_{L^2(\bt^d)}^2 .
\ee
This completes the proof of \eqref{H1dot}.

\end{proof}

\begin{proof}[Proof of Proposition~\ref{prop:Ustab}]
Estimate \eqref{stab-Ubar} follows directly from  Lemma~\ref{lem:Loep}.
The only remaining task is to prove \eqref{stab-Uhat}. We want to apply Lemma~\ref{lem:hatU-stab}, which requires $L^\infty(\bt^d)$ estimates on $\bar U_i $ and $\widehat U_i $ ($i=1,2$). By Proposition~\ref{prop:regU},
\be 
\lVert  \bar U _i \rVert_{L^\infty(\bt^d)}  \leq C_{d}    \left (1 +  \lVert h_i \rVert_{L^{\frac{d+2}{d}}(\bt^d)} \right ), \qquad
\|\widehat U _i \|_{L^{\infty}(\bt^d)}  \leq \exp{ \left (C_{d}     \(1 +  \lVert h_i \rVert_{L^{\frac{d+2}{d}}(\bt^d)} \)  \right)}  .
\ee
Hence, by Lemma~\ref{lem:hatU-stab}, we obtain
\be 
\lVert \nabla \widehat{U}_1  - \nabla \widehat{U}_2  \rVert^2_{L^2(\bt^d)} \leq C  \lVert \bar{U}_1  - \bar{U}_2  \rVert^2_{L^2(\bt^d)} ,
\ee
with
\be
C  \leq \,\exp\,\exp   { \left [C_{d}     \left (1 +  \max_i\, \lVert h_i \rVert_{L^{\frac{d+2}{d}}(\bt^d)} \right )  \right]} .
\ee
The Poincar\'e inequality for zero mean functions implies that
\be 
\lVert \nabla \widehat{U}_1  - \nabla \widehat{U}_2  \rVert^2_{L^2(\bt^d)} \leq C  \lVert \bar{U}_1  - \bar{U}_2  \rVert^2_{L^2(\bt^d)}\leq C  \lVert \nabla \bar{U}_1  - \nabla \bar{U}_2  \rVert^2_{L^2(\bt^d)} .
\ee
Hence by \eqref{stab-Ubar},
\be 
\lVert \nabla \widehat{U}_1  - \nabla \widehat{U}_2  \rVert^2_{L^2(\bt^d)} \leq C  \max_i\, \lVert h_i \rVert_{L^{\infty}(\bt^d)} \, W^2_2(h_1, h_2) ,
\ee
where $C $ may be chosen to satisfy
\be
C  \leq \,\exp\,\exp   { \left (C_{d}     \left (1 +  \max_i\, \lVert h_i \rVert_{L^{\frac{d+2}{d}}(\bt^d)} \right )  \right)} 
\ee
for some suitably large $C_d$.

\end{proof}

\section{Uniqueness for Solutions With Bounded Density} \label{sec:stab}

This section focuses on the uniqueness part of Theorem~\ref{thm:main}. The aim is to prove the following theorem, concerning the uniqueness of solutions for the VPME system under the condition that the mass density $\rho_f$ is bounded in $L^\infty(\TT^d)$.

\begin{thm}[Uniqueness for solutions with bounded density]
\label{thm:unique}
Let $d = 2, 3$. Let ${f_0 \in \mc{P}(\bt^d \times \br^d)}$ with $\rho_{f_0} \in L^\infty(\bt^d)$. Fix a final time $T > 0$. Then there exists at most one solution ${f\in C([0,T] ; \mc{P}(\bt^d \times \br^d))}$ of \eqref{vpme} with initial datum $f_0$ such that $\rho_f \in L^{\infty}([0,T] ; L^\infty(\bt^d))$. 

\end{thm}

The proof of this result is based on a stability estimate on solutions of the VPME system \eqref{vpme} with respect to the initial datum. 
This stability estimate is quantified using Wasserstein distances.

\subsection{Strong-strong stability}

In this section, we prove a quantitative stability estimate in $W_2$ between solutions with bounded \mbox{density}\footnote{Such estimates are said to be of \em{strong-strong}-type because we are requiring that both densities belong to $L^\infty$.}. To do this we will make use of the stability estimates for the electric field that we have proved in section~\ref{sec:electric}. Following the decomposition \eqref{electric-field}, it is useful to rewrite \eqref{vpme} in the form
$$
(VPME) := \left\{ \begin{array}{ccc}\pt_t f +v\cdot \nabla_x f + (\bar E +\widehat E )\cdot \nabla_v f =0,  \\
\bar E =-\nabla \bar U ,\qquad \widehat E =-\nabla \widehat U , \\
 \Delta\bar U =1-\rho_f ,\\
 \Delta\widehat U =e^{\bar U +\widehat U } - 1,\\
\rho_f  = \int_{\br^d} f  \di v \\
f_0 (x,v)\ge0,\ \  \int_{\bt^d \times \br^d} f_0 (x,v)\,dx\,dv=1.
\end{array} \right.
$$
The aim is to prove the following estimate between two solutions of VPME \eqref{vpme} that have bounded density.

\begin{prop}[Stability for solutions with bounded density] \label{prop:Wstab}
For $i = 1,2$, let $f _i$ be solutions of \eqref{vpme} satisfying for some constant $M$ and all $t \in [0,T]$,
\be \label{str-str_rho-hyp}
\rho[f _i(t)] \leq M .
\ee
Then there exists a constant $C $, depending on $M$, such that, for all $t \in [0,T]$,
\be
W_2\(f _1(t), f _2(t)\)^2 \leq \begin{cases}  \frac{d e}{4} \exp{\left ( \log{\frac{4 W_2\(f _1(0), f _2(0)\)^2 }{de}} \, e^{- C  t} \right ) }  & \mbox{if}\, \,t \leq t_0 \\
\max \left \{ \frac{d}{4} , W_2\(f _1(0), f _2(0)\)^2  \right \} e^{C (t-t_0)} &\mbox{if}\, \, t > t_0 ,
\end{cases}\ee
where the time $t_0$ is defined by
\be
t_0 = t_0\big(W_2\(f _1(0), f _2(0)\) \big) = \inf \left\{ t \geq 0 :  \frac{d e}{4} \exp{\left ( \log{\frac{4 W_2\(f _1(0), f _2(0)\)^2 }{de}} \, e^{- C  t} \right ) }  > \frac{d}{4} \right\} .
\ee

\end{prop}
Theorem~\ref{thm:unique} then follows from this estimate.

\begin{proof}
We will prove Proposition~\ref{prop:Wstab} by means of a Gronwall type estimate. To do this, we will first consider the evolution of particular specially constructed couplings $\pi_t \in \Pi(f_1(t), f_2(t))$. First, observe that $f_i $ can be represented as the pushforward of the initial datum $f_i (0)$ along the characteristic flow associated to \eqref{vpme}. That is, given $f_i $, consider for each $(x,v) \in \bt^d \times \br^d$ the system of ODEs
\begin{align} \label{eqn:char}
& \begin{cases} \dot X^{(i)}_{x,v} = V^{(i)}_{x,v} \\ 
\dot V^{(i)}_{x,v}  = E_i (X^{(i)}_{x,v})
\end{cases} \\
& (X^{(i)}_{x,v}(0), V^{(i)}_{x,v}(0)) = (x, v) .
\end{align}
where $E _i$ is the electric field induced by $f_i $:
\be
E _i=-\nabla U_i , \qquad \Delta U _i=e^{U _i}- \rho[f_i ] .
\ee
We again use the decomposition $E _i= \widehat E _i + \bar E _i$. Since $\rho[f_i ] \in L^\infty(\bt^d)$ by assumption \eqref{str-str_rho-hyp}, Lemma~\ref{lem:logLip} implies that $\bar E _i$ has log-Lipschitz regularity. Since $L^\infty(\bt^d) \subset L^{\frac{d+2}{d}}(\bt^d)$, we have $\rho[f_i ] \in L^\infty \cap L^{\frac{d+2}{d}}(\bt^d)$. Thus we may apply Proposition~\ref{prop:regU} to deduce Lipschitz regularity of $\widehat E _i$. Overall this implies that $E _i$ has log-Lipschitz regularity, which is sufficient to guarantee the existence of a unique solution to the system \eqref{eqn:char}. The uniqueness of the flow implies that the linear Vlasov equation
\be \label{eqn:vlas-lin}
\partial_t g + v \cdot \nabla_x g + E _i \cdot \nabla_v g = 0, \qquad
g\vert_{t=0} = f _i(0)
\ee
has a unique measure-valued solution $g$  (see for instance \cite[Theorem 3.1]{Amb}). This solution can be represented as the pushforward of the initial data along the characteristic flow, which means that $g_t$ satisfies
\be \label{f-rep}
\int_{\bt^d \times \br^d} \phi(x,v) g_t(\di x \di v) = \int_{\bt^d \times \br^d} \phi \left (X^{(i)}_{x,v}, V^{(i)}_{x,v} \right ) f_i(0, x,v) \di x \di v 
\ee
for all $\phi \in C_b(\bt^d \times \br^d)$. Since $f_i $ is also a solution of \eqref{eqn:vlas-lin}, and the solution is unique, it follows that $g = f_i $. We deduce that $f_i $ has the representation \eqref{f-rep}. Note that here we are not yet asserting any \textit{nonlinear} uniqueness, because we already fixed $E_i $ to be the electric field corresponding to $f_i $.

We use the representation above to construct $\pi_t$. First, fix an arbitrary initial coupling $\pi_0 \in \Pi \left (f_1 (0), f_2 (0) \right )$. We then build a coupling $\pi_t$ for which each marginal evolves along the appropriate characteristic flow. To be precise, we define $\pi_t$ to be the measure such that, for all $\phi \in C_b((\bt^d \times \br^d)^2)$,
\be \label{def:pit}
\int_{(\bt^d \times \br^d)^2} \phi(x,v,y,w) \di \pi_t(x,v,y,w) = \int_{\bt^d \times \br^d} \phi \left (X^{(1)}_{x,v}, V^{(1)}_{x,v}, X^{(2)}_{y,w}, V^{(2)}_{y,w} \right ) \di \pi_0(x,v,y,w) .
\ee
By checking the marginals:
\begin{align} \label{pit-marg}
\int_{(\bt^d \times \br^d)^2} \phi(x_{i},v_{i}) \di \pi_t(x_1,v_1,x_2,v_2) & = \int_{(\bt^d \times \br^d)^2} \phi \left (X^{(i)}_{x_i,v_i}, V^{(i)}_{x_i,v_i}\right ) \di \pi_0(x_1,v_1,x_2,v_2) \\
& = \int_{\bt^d \times \br^d} \phi \left (X^{(i)}_{x_i,v_i}, V^{(i)}_{x_i,v_i}\right ) f_i (0, x_i, v_i) \di x_i \di v_i \\
& = \int_{\bt^d \times \br^d} \phi(x,v) f_i (t, x, v) \di x \di v ,
\end{align}
we see that the representation \eqref{f-rep} implies that $\pi_t \in \Pi\left( f_1(t) , f_2(t) \right )$ for all $t \in [0,T]$.

\noindent We now consider the quantity
\be \label{def:D}
D(t) : = \int |X^{(1)}_t - X^{(2)}_t|^2 +  |V^{(1)}_t - V^{(2)}_t|^2 \di \pi_0 .
\ee
We have omitted the subscripts $x,v,y,w$ in order to lighten the notation. Since by definition \eqref{def:pit} we have
\be
D(t) = \int_{(\bt^d \times \br^d)^2} |x-y|^2 + |v-w|^2 \di \pi_t ,
\ee
it follows from Definition~\ref{def:Wass} that
\be \label{D-ctrl-W}
W_2^2\(f_1 (t), f_2 (t)\) \leq C D(t) .
\ee
Moreover, since $\pi_0$ was arbitrary, we have
\be \label{W0-ctrl-D0}
W_2^2\(f_1 (0), f_2 (0)\) = \inf_{\pi_0} D(0) .
\ee
We will therefore focus next on controlling the growth of $D(t)$. This amounts to performing a Gronwall estimate along the trajectories of the characteristic flow. We give the details in Lemma~\ref{lem:D-est} below. We obtain a bound
\be
D(t) \leq
\begin{cases}  \frac{d e}{4} \exp{\left ( \log{\frac{4 D(0)}{de}} \, e^{- C  t} \right ) }  & \mbox{if}\, \,t \leq t_0 \\
\max \left \{ \frac{d}{4} , D(0) \right \} e^{C (t-t_0)} &\mbox{if}\, \, t > t_0 ,
\end{cases}
\ee
where $t_0$ is defined by
\be \label{def:t0}
t_0 = t_0\big(D(0)\big) = \inf \left\{ t \geq 0 :  \frac{d e}{4} \exp{\left ( \log{\frac{4 D(0)}{de}} \, e^{- C  t} \right ) }   > \frac{d}{4} \right\} .
\ee
Observe that $t_0$ is decreasing as a function of $D(0)$.
From \eqref{D-ctrl-W} it follows that
\be
W_2^2\(f_1 (t), f_2 (t)\) \leq
\begin{cases}  \frac{d e}{4} \exp{\left ( \log{\frac{4 D(0)}{de}} \, e^{- C  t} \right ) }  & \mbox{if}\, \,t \leq t_0 \\
\max \left \{ \frac{d}{4} , D(0) \right \} e^{C (t-t_0)} &\mbox{if}\, \, t > t_0 ,
\end{cases}
\ee
Finally, taking infimum over $\pi_0$ and applying \eqref{W0-ctrl-D0} concludes the proof.
\end{proof}

\begin{lem}[Control of $D$] \label{lem:D-est} Let $D$ be defined by \eqref{def:D}. Then
\be
D(t) \leq
\begin{cases}  \frac{d e}{4} \exp{\left ( \log{\frac{4 D(0)}{de}} \, e^{- C  t} \right ) }  & \mbox{if}\, \,t \leq t_0 \\
\max \left \{ \frac{d}{4} , D(0) \right \} e^{C (t-t_0)} &\mbox{if}\, \, t > t_0 ,
\end{cases}
\ee
where $C $ depends on $M$ and $t_0$ is defined by \eqref{def:t0}. 
\end{lem}
\begin{proof}

Differentiating with respect to $t$ gives
\be \label{D-derivative}
\dot D(t) = 2 \int_{(\bt^d \times \br^d)^2} (X^{(1)}_t - X^{(2)}_t) \cdot (V^{(1)}_t - V^{(2)}_t ) + (V^{(1)}_t - V^{(2)}_t ) \cdot \[ E _1(X^{(1)}_t) - E_2 (X^{(2)}_t) \] \di \pi_0 
\ee
We split the electric field into four parts:
\begin{align}
E _1(X^{(1)}_t) - E_2 (X^{(2)}_t) & = \[ \bar E _1(X^{(1)}_t) - \bar E_1 (X^{(2)}_t) \] + \[ \bar E _1(X^{(2)}_t) - \bar E_2 (X^{(2)}_t) \]  \\
& \qquad + \[ \widehat E _1(X^{(1)}_t) - \widehat E_1 (X^{(2)}_t) \] + \[ \widehat E _1(X^{(2)}_t) - \widehat E_2 (X^{(2)}_t)\] .
\end{align}
Applying H\"older's inequality to \eqref{D-derivative}, we obtain
\be
\dot D \leq D + 2 \sqrt{D} \sum_{i=1}^4 I_i^{1/2} ,
\ee
where
\begin{align} \label{def:Ii}
\begin{split}
&I_1 := \int_{(\bt^d \times \br^d)^2} |\bar E _1(X^{(1)}_t) - \bar E_1 (X^{(2)}_t)|^2 \di \pi_0, \quad I_2 := \int_{(\bt^d \times \br^d)^2} |\bar E _1(X^{(2)}_t) - \bar E_2 (X^{(2)}_t)|^2 \di \pi_0; \\
&I_3:= \int_{(\bt^d \times \br^d)^2} |\widehat E _1(X^{(1)}_t) - \widehat E_1 (X^{(2)}_t)|^2 \di \pi_0, \quad I_4 := \int_{(\bt^d \times \br^d)^2} |\widehat E _1(X^{(2)}_t) - \widehat E_2 (X^{(2)}_t)|^2 \di \pi_0.
\end{split}
\end{align}
We estimate the above terms in Lemmas~\ref{lem:I1}-\ref{lem:I4} below. Altogether we obtain
\be
\dot D \leq \begin{cases} C  D \left (1 + \lvert \log{\frac{4 D}{ d}} \rvert \right ) & \text{ if } D < \frac{d}{4} \\
 C  D & \text{ if } D \geq \frac{d}{4} .
\end{cases}
\ee
Therefore
\be
D(t) \leq  \frac{d e}{4} \exp{\left ( \log{\frac{4 D(0)}{de}} \, e^{- C  t} \right ) } 
\ee
as long as $D(t) \leq \frac{d}{4}$, which certainly holds as long as $t < t_0$. For $t > t_0$ we have the alternative bound
$$
D(t) \leq \max \left \{ \frac{d}{4} , D(0) \right \} \,e^{C (t- t_0)} .
$$
\end{proof}

\begin{lem}[Control of $I_1$] \label{lem:I1}
Let $I_1$ be defined as in \eqref{def:Ii}. Then
\be
I_1 \leq C_d (M+1)^2 H(D),
\ee
where $D$ is defined as in \eqref{def:D} and
$$
H(x) := \begin{cases}
x \left (\log{\frac{4x}{e^2 d}} \right )^2 &\text{ if } x \leq \frac{d}{4} \\
d &\text{ if } x > \frac{d}{4}.
\end{cases}
$$
\end{lem}
\begin{proof}
First we use the regularity estimate for $\bar E_1 $ from Lemma~\ref{lem:logLip}:
\begin{align}
I_1 & \leq  C_d \lVert \rho_1  - 1 \rVert_{L^{\infty}(\bt^d)}^2 \int_{(\bt^d \times \br^d)^2} |X^{(1)}_t - X^{(2)}_t|^2 \left ( \log{\frac{e \sqrt{d}}{2|X^{(1)}_t - X^{(2)}_t|}} \right )^2 \di \pi_0 \\
& = C_d   \lVert \rho_1  - 1 \rVert_{L^{\infty}(\bt^d)}^2 \int_{(\bt^d \times \br^d)^2} |X^{(1)}_t - X^{(2)}_t|^2 \left ( \log{\frac{4 |X^{(1)}_t - X^{(2)}_t|^2}{e^2 d}} \right )^2 \di \pi_0 .
\end{align}
The function 
$$a(x) = x \left (\log{\frac{4 x}{e^2 d }} \right )^2$$
 is concave on the set $x \in [0, \frac{d e}{4}]$. Since $X^{(i)}_t \in \bt^d$, we have $|X^{(1)}_t - X^{(2)}_t|^2 \leq \frac{d}{4}$. Note that
$$
a ' \left (\frac{d}{4} \right) = - \log{e^2} (2 - \log{e^2}) = 0 ;
$$
hence the function $H(x)$ defined in the statement is concave on $\mathbb R^+$, and
\be
I_1 \leq C_d   \lVert \rho_1  - 1 \rVert_{L^{\infty}(\bt^d)}^2 \int_{(\bt^d \times \br^d)^2} H\,( |X^{(1)}_t - X^{(2)}_t|^2)  \di \pi_0 .
\ee
Then, since $\pi_0$ is a probability measure, we may apply Jensen's inequality to deduce that
\be
I_1  \leq C_d  \lVert \rho_1  - 1 \rVert_{L^{\infty}(\bt^d)}^2 \, \, H\left( \int_{(\bt^d \times \br^d)^2} |X^{(1)}_t - X^{(2)}_t|^2 \di \pi_0 \right) \leq C_d  \lVert \rho_1  - 1 \rVert_{L^{\infty}(\bt^d)}^2 \, \, H(D).
\ee
\end{proof}

\begin{lem}[Control of $I_2$] \label{lem:I2}
Let $I_2$ be defined as in \eqref{def:Ii}. Then
\be
I_2 \leq  M^2 D ,
\ee
where $D$ is defined as in \eqref{def:D}.
\end{lem}
\begin{proof}
From \eqref{pit-marg}, for all $\phi \in C(\bt^d)$ we have
\be
\int_{(\bt^d \times \br^d)^2} \phi(X^{(i)}_t) \di \pi_0 = \int_{\bt^d \times \br^d} \phi(x) f_i (t, x,v) \di x \di v \label{pushforward2}= \int_{\bt^d} \phi(x) \rho_i (t, x) \di x .
\ee
Thus
\begin{align}
I_2 & = \int_{\bt^d} |\bar E _1(x) - \bar E_2 (x)|^2 \rho_2 (t,x) \di x \leq \lVert \rho_2  \rVert_{L^{\infty}(\bt^d)} \lVert \bar E _1 - \bar E_2  \rVert_{L^2(\bt^d)}^2 = \lVert \rho_2  \rVert_{L^{\infty}(\bt^d)} \lVert \nabla \bar U _1 - \nabla \bar U_2  \rVert_{L^2(\bt^d)}^2.
\end{align}
We use the stability estimate from Lemma~\ref{lem:Loep} to control the difference between different electric fields:
\be
I_2 \leq   \max_i\, \lVert \rho _i \rVert_{L^{\infty}(\bt^d)}^2 \, W_2^2(\rho _1, \rho _2) \leq   \max_i\, \lVert \rho _i \rVert_{L^{\infty}(\bt^d)}^2 \, D .
\ee

\end{proof}

\begin{lem}[Control of $I_3$] \label{lem:I3}
Let $I_3$ be defined as in \eqref{def:Ii}. Then
\be
I_3 \leq C_{M,d}  \, D ,
\ee
where $D$ is defined as in \eqref{def:D} and $C_{M,d}$ depends on $M$ and $d$.
\end{lem}
\begin{proof}
Observe that
\begin{align}
I_3 & = \int_{(\bt^d \times \br^d)^2} |\widehat E _1(X^{(1)}_t) - \widehat E_1 (X^{(2)}_t)|^2 \di \pi_0 \\
& \leq \int_{(\bt^d \times \br^d)^2} \lVert \widehat E _1 \rVert_{C^{1}(\bt^d)}^2 |X^{(1)}_t - X^{(2)}_t|^2 \di \pi_0 \leq \lVert \widehat U _1\rVert_{C^{2,\alpha}(\bt^d)}^2 D 
\end{align}
for any $\alpha>0$. To this we apply the regularity estimate on $\widehat U_1$ from Proposition~\ref{prop:regU}
with  $\alpha \in A_d$:
\be
\|\widehat U_1  \|_{C^{2,\alpha}(\bt^d)} \leq C_{\alpha,d} \,\exp\,\exp  {\(C_{\alpha, d}     \(1 +  \lVert \rho_1  \rVert_{L^{\frac{d+2}{d}}(\bt^d)}\) \)} \leq C_{\alpha, M ,d} ,
\ee
since
\be
\| \rho_i \|_{L^{\frac{d+2}{d}}(\bt^d)} \leq \| \rho_i \|_{L^\infty(\bt^d)} \leq M .
\ee
Thus we have
\be
I_3 \leq C_{M,d} \, D .
\ee

\end{proof}

\begin{lem}[Control of $I_4$] \label{lem:I4}
Let $I_4$ be defined as in \eqref{def:Ii}. Then
\be
I_4  \leq C_{M,d} \, D,
\ee
where $D$ is defined as in \eqref{def:D} and $C_{M,d}$ depends on $M$ and $d$.
\end{lem}
\begin{proof}
Using \eqref{pushforward2} again, we deduce that
\begin{align}
I_4 &  = \int_{\bt^d} |\widehat E _1(x) - \widehat E_2 (x)|^2 \rho _2(t,x) \di x \\
& \leq \lVert \rho _2 \rVert_{L^{\infty}(\bt^d)} \lVert \widehat E _1 - \widehat E _2 \rVert^2_{L^2(\bt^d)}
 = \lVert \rho _2 \rVert_{L^{\infty}(\bt^d)} \lVert \nabla \widehat U _1 - \nabla \widehat{U} _2 \rVert^2_{L^2(\bt^d)} .
\end{align}
To control the $L^2(\bt^d)$ distance between the electric fields we use the stability estimate in Proposition~\ref{prop:Ustab}:
\be
\lVert \nabla \widehat{U}_1  - \nabla \widehat{U}_2  \rVert^2_{L^2(\bt^d)} \leq \exp\,\exp {\left [ C_d    \left ( 1 + \max_i\, \lVert \rho _i \rVert_{L^{(d+2)/d}(\bt^d)}  \right ) \right ]} \max_i\, \lVert \rho _i \rVert_{L^{\infty}} \, W_2^2(\rho _1, \rho _2) .
\ee
Therefore
\begin{align}
I_4 & \leq \exp\,\exp {\left [ C_d    \left ( 1 + \max_i\, \lVert \rho _i  \rVert_{L^{(d+2)/d}(\bt^d)}  \right ) \right ]} \max_i\, \lVert \rho _i \rVert^2_{L^{\infty}} \, W_2^2(\rho _1, \rho _2) \\
& \leq \exp\,\exp {\left [ C_d    \left ( 1 + \max_i\, \lVert \rho _i \rVert_{L^{(d+2)/d}(\bt^d)}  \right ) \right ]} \max_i\, \lVert \rho _i \rVert^2_{L^{\infty}} \, D \leq C_{M,d} \,D .
\end{align}

\end{proof}

\section{Propagation of Moments} \label{sec:moments}

In this section we prove an a priori estimate on classical solutions of the VPME system, showing that velocity moments of sufficiently high order are propagated. At this stage, the reader may look at all the computations in this section as a priori estimates for classical solutions that decay fast enough at infinity. More precisely, we shall prove uniform moment propagation estimates for $C^1$ compactly supported solutions, with bounds that are independent of the smoothness of $f$ and of the fact that $f$ has compact support.

Then, in Section~\ref{sec:construction}, we will perform the same estimates on a family of solutions of regularised  VPME systems where all the computations will be rigorous. Of course one could have performed these estimates directly on the regularised systems. However this choice simplifies the notation and highlights the main ideas.

Note that a posteriori, as a consequence of our main Theorem \ref{thm:main} and Remark \ref{rem:C1}, $C^1$ compactly supported solutions of 
VPME system \eqref{vpme} exist whenever the initial datum is $C^1$ and compactly supported.

\begin{prop} \label{prop:moment-propagation}
Let the dimension $d=2$ or $d=3$. Let $0 \leq f_0 \in L^1 \cap L^\infty(\TT^d \times \RR^d)$ have a finite energy and finite velocity moment of order $m_0 > d$:
\be
\mc{E} [f] \leq C_0<+\infty,\qquad \int_{\TT^d \times \RR^d} |v|^{m_0} f_0(x,v) \di x \di v =M_0< +\infty.
\ee
Let $f$ be a $C^1$ compactly supported solution of the VPME system \eqref{vpme}.
Then, for all $T > 0$,
\be
\sup_{[0,T]} \int_{\TT^d \times \RR^d} |v|^{m_0} f(t,x,v) \di x \di v \le C(T, C_0,M_0,m_0,\|f_0\|_\infty).
\ee
\end{prop}

Our approach is based on adapting known methods for the Vlasov-Poisson system \eqref{vp} to the VPME case. The methods differ according to the dimension $d$.
In the two dimensional case $d=2$, we follow the strategy explained in \cite[Section 4.3]{Golse-notes}.
In the three dimensional case, we adapt techniques by Pallard \cite{Pallard} and Chen and Chen \cite{Chen-Chen}.

\subsection{Interpolation Estimate} \label{sec:rho-interpolation}

The following interpolation result allows quantities such as the mass density $\rho_f$ to be estimated in terms of the velocity moments. We introduce the notation $M_m$ for the moment of order $m>0$:
\be \label{def:Mm}
M_m(t) : = \int_{\TT^d \times \RR^d} |v|^m f(t,x,v) \di x \di v .
\ee
{Let us start with a classical Lemma \ref{lem:mom-interpolation}, that we prove for the convenience of the reader.}
\begin{lem} \label{lem:mom-interpolation}
Let $g\geq 0$ be a function in $L^\infty(\TT^d \times \RR^d)$. Assume that $M_m$ as defined in \eqref{def:Mm} is finite for some $m > 0$.
For $k \in [0,m]$, consider the local velocity moments
\be \label{def:lk}
l_k(x) : = \int_{\RR^d} |v|^k g(x,v) \di v .
\ee
There exists a constant $C_{d,m,k}$ depending on $m,k$ and $d$ such that
\be 
\lVert l_k \rVert_{L^{\frac{m+d}{k+d}}(\bt^d)} \leq C_{d,m,k} \lVert g \rVert_{L^{\infty}(\bt^d \times \br^d)}^{\frac{m-k}{m+d}} \lVert l_m\rVert_{L^1}^{\frac{k+d}{m+d}}.
\ee

\end{lem}
\begin{proof}

 Fix $x \in \bt^d$ and split the integral defining $l_k$ in \eqref{def:lk} into a part close to zero and a part far from zero. We obtain, for arbitrary $R > 0$,
\begin{align}
l_k(x) & = \int_{|v| \leq R} |v|^k g(x,v) \di v + \int_{|v| > R} |v|^k g(x,v) \di v  \leq \int_{|v| \leq R} |v|^k g(x,v) \di v + R^{k-m} \int_{|v| > R} g(x,v) |v|^m \di v \\
& \leq \lVert g(x, \cdot) \rVert_{L^{\infty}(\br^d)}\, C_d R^{k+d} + R^{k-m} l_m(x),
\end{align}
where $C_d$ is a dimension dependent constant. The optimal choice of $R$ is 
$$
R(x) = C_{d,m,k} \left ( \frac{l_m(x)}{\lVert g(x, \cdot) \rVert_{L^{\infty}(\br^d)}} \right )^{1/(d+m)}
$$
which results in the estimate
\be 
l_k(x) \leq C_{d,m,k} \lVert g(x, \cdot) \rVert_{L^{\infty}(\br^d)}^{\frac{m-k}{m+d}} l_m(x)^{\frac{k+d}{m+d}} \leq C_{d,m,k} \lVert g \rVert_{L^{\infty}(\br^d)}^{\frac{m-k}{m+d}} l_m(x)^{\frac{k+d}{m+d}}.
\ee

Thus
\be 
\lVert l_k \rVert_{L^{\frac{m+d}{k+d}}(\bt^d)} \leq C_{d,m,k} \lVert g\rVert_{L^{\infty}(\bt^d \times \br^d)}^{\frac{m-k}{m+d}} \lVert l_m\rVert_{L^1}^{\frac{k+d}{m+d}}.
\ee
\end{proof}

In particular, if a solution of the VPME system \eqref{vpme} has bounded energy, then its mass density satisfies a certain $L^p$ estimate.

\begin{lem} \label{lem:rho-Lp}
Let $g \geq 0$ satisfy, for some constant $C_0$,
\be
\lVert g \rVert_{L^{\infty}(\bt^d \times \br^d)} \leq C_0,\qquad
\mc{E} [g] \leq C_0 ,
\ee
where $\mc{E}$ is the energy functional defined in \eqref{def:Ee}. Then,
\be \label{f-mom2}
\int_{\bt^d \times \br^d} |v|^2 g \di x \di v \leq C_1,
\ee
for some constant $C_1$ depending on $C_0$ only. Moreover the mass density
\be \label{def:rho}
\rho_g(x) : = \int_{\br^d} g(x,v) \di v
\ee
lies in $L^{(d+2)/d}(\bt^d)$ with
\be \label{rho-Lp}
\lVert \rho_g \rVert_{L^{\frac{d+2}{d}}(\bt^d)} \leq C_2 .
\ee
for some constant $C_2$ depending on $C_1$ and $d$ only.
\end{lem}

\begin{proof}
	Recall that
	\be 
	\mc{E} [g]:= \frac{1}{2}\int_{\bt^d \times \br^d} |v|^2 g \di x \di v + \frac{ 1 }{2} \int_{\bt^d} |\nabla U |^2 \di x +  \int_{\bt^d} U  e^{U } \di x,
	\ee
	where $U$ solves $-\Delta U=e^U-\rho_g$.
	The moment estimate \eqref{f-mom2} follows from the fact that for all $x \in \br$, $x e^x \geq - e^{-1}.$
	Hence the boundedness of $\mc{E} [f]$ implies that
	\be
	\int_{\bt^d \times \br^d} |v|^2 g \di x \di v \leq 2 (C_0 + e^{-1} ) .
	\ee
	The estimate \eqref{rho-Lp} on $\rho_g$ then follows from Lemma~\ref{lem:mom-interpolation}, upon choosing $m=2$ and $k=0$.

\end{proof}

\subsection{Two Dimensions}

In this subsection we always take the dimension $d=2$. The goal is to prove the following lemma on the propagation of moments in two dimensions.

\begin{lem}
	Let $f_0 \in L^1 \cap L^\infty(\TT^2 \times \RR^2)$ satisfy $\mc{E} [f] \leq C_0<+\infty$ and
	\be 
M_{m_0}(0) :=	\int_{\TT^2 \times \RR^2} |v|^{m_0} f_0(x,v) \di x \di v  < +\infty 
	\ee 
	for some $m_0 > 2$.
	Let $f$ be a solution of the VPME system \eqref{vpme}  as in Proposition \ref{prop:moment-propagation}.
	There exists a constant $C_{m_0,0}$, depending only on $m_0$, $M_{m_0}(0)$, $C_0$, and $\|f_0\|_\infty,$ such that
	\be
	M_{m_0}(t) \leq C_{m_0,0} (1+t)^{m_0+2}.
	\ee
\end{lem}
\begin{proof}
Using the pushforward representation of $f$, we have, for all $t \in [0,T]$,
\begin{align}
M_{m_0}(t) &= \int_{\mc{Q}_2 \times \RR^2} |v|^{m_0} f(t,x,v) \di x \di v \\
& = \int_{\mc{Q}_2 \times \RR^2} |V(t ; 0,x,v)|^{m_0} f_0(x,v) \di x \di v .
\end{align}
This identity can be used to calculate the time derivative of $M_{m_0}$, using the definition of the characteristic flow: for any ${m_0} > 2$,
\begin{align}
\frac{\di }{\di t} M_{m_0}(t) & = {m_0} \int_{\mc{Q}_2 \times \RR^2} E \left (X(t; 0,x,v) \right ) \cdot V(t ; 0,x,v) |V(t ; 0,x,v)|^{{m_0}-2} f_0(x,v) \di x \di v \\
& = {m_0} \int_{\mc{Q}_2 \times \RR^2} E (x) \cdot v |v|^{{m_0}-2} f(t,x,v) \di x \di v \\
& \leq {m_0} \int_{\mc{Q}_2 } |E (x)|  \int_{\RR^2} |v|^{{m_0}-1} f(t,x,v) \di v \di x  = {m_0} \int_{\mc{Q}_2 } |E (x)| \, l_{{m_0}-1}(x) \di x  .
\end{align}
Applying H\"{o}lder's inequality in the $x$ variable with exponent $p$ gives
\be
\frac{\di }{\di t} M_{m_0}(t) \leq {m_0} \| E \|_{L^{p'}(\TT^2)} \left \| l_{{m_0}-1} \right \|_{L^p(\TT^2)} .
\ee
Choose $p = \frac{{m_0}+2}{{m_0}+1}$ and apply the moment interpolation estimate from Lemma~\ref{lem:mom-interpolation}, to obtain
\be
\frac{\di }{\di t} M_{m_0}(t) \leq C_{m_0} \| E \|_{L^{{m_0}+2}(\TT^2)} M_{m_0}(t)^{\frac{{m_0}+1}{{m_0}+2}}
\ee
where the constant $C_{m_0} > 0$ depends only on ${m_0}$.
The remaining step is to estimate $\| E \|_{L^{m_0+2}(\TT^2)}$. To do this, first use the decomposition $E = \bar E + \widehat E$:
\be
\| E \|_{L^{{m_0}+2}(\TT^2)} \leq \| \bar E \|_{L^{{m_0}+2}(\TT^2)} + \| \widehat E \|_{L^{{m_0}+2}(\TT^2)} .
\ee
For the Poisson part $\bar E$, standard regularity estimates for the Poisson equation imply that
\be
\| \bar E \|_{L^{{m_0}+2}(\TT^2)} \leq C_{m_0} \| \rho_f \|_{L^{\frac{2({m_0}+2)}{{m_0}+4}}(\TT^2)}
\ee
Note that $\frac{2({m_0}+2)}{{m_0}+4} < 2$. Thus, due to the uniform control of the energy functional $\mc{E}$ (defined in \eqref{def:Ee}) we have
\be
\| \bar E \|_{L^{{m_0}+2}(\TT^2)} \leq C_{m_0} \| \rho_f \|_{L^2(\TT^2)} \leq C_{{m_0},0},
\ee
where $C_0 > 0$ depends only on ${m_0}$, $\mc{E}[f_0]$ and $\| f_0 \|_{L^\infty(\TT^2 \times \RR^2)}$. For $\widehat E$, we use Proposition~\ref{prop:regU} which provides the bound
\be
\| \widehat E \|_{L^\infty(\TT^2)} \leq \exp \left ( C_d \left (1 + \| \rho_f\|_{L^2(\TT^2)} \right ) \right ) \leq C_0,
\ee
where $C_0 > 0$ depends only on $\mc{E}[f_0]$ and $\| f_0 \|_{L^\infty(\TT^2 \times \RR^2)}$ (here we are using Lemma \ref{lem:rho-Lp} and the fact that our solution conserves the energy). Altogether, we have the estimate
\be
\| E \|_{L^{{m_0}+2}(\TT^2)} \leq C_{{m_0},0} ,
\ee
and thus
\be
\frac{\di }{\di t} M_{m_0}(t) \leq C_{{m_0},0} M_{m_0}(t)^{\frac{{m_0}+1}{{m_0}+2}}.
\ee

It follows that there exists a constant $C_{{m_0},0} > 0$ depending on ${m_0}$, $\mc{E}[f_0]$, $\| f_0 \|_{L^\infty(\TT^2 \times \RR^2)}$ and $M_{m_0}(0)$ such that for all $t \in [0,T]$,
\be
M_{m_0}(t) \leq C_{{m_0},0} (1+t)^{{m_0}+2}.
\ee
\end{proof}

\subsection{Three Dimensions}

In this subsection, we prove propagation of moments in the three dimensional case $d=3$, which is stated in the lemma below.

\begin{lem} \label{lem:moment-propagation-3d}
Let $d=3$. Let $f$ be a solution of the VPME system \eqref{vpme}  as in Proposition \ref{prop:moment-propagation}, with initial datum $f_0 \in L^1 \cap L^\infty(\TT^d \times \RR^d)$ satisfying
\be
\int_{\TT^d \times \RR^d} |v|^{m_0} f_0(x,v) \di x \di v =M_0< + \infty, \qquad m_0 > d .
\ee
Then, for all $T>0$, \be
\sup_{t \in [0,T]} \int_{\TT^d \times \RR^d} |v|^{m_0} f(t,x,v) \di x \di v \leq C(T,M_0,m_0,\|f_0\|_\infty) .
\ee
\end{lem}

As in the two dimensional case, the time evolution of the moment $M_{m_0}(t)$ can be studied by using the pushforward representation of $f$:
\be
\int_{\TT^d \times \RR^d} |v|^{m_0} f(t,x,v) \di x \di v = \int_{\TT^d \times \RR^d} |V(t; 0,x,v)|^{m_0} f_0(x,v) \di x \di v.
\ee
Observe that
\be
|V(t; 0,x,v)| \leq |v| + \left | \int_0^t E \left ( X(t;0,x,v) \right ) \di s \right | .
\ee
The next step is to estimate $E$.
As was discussed for the two dimensional case, the overall strategy is to use the decomposition $E = \bar E + \widehat E$, and to notice that by Proposition~\ref{prop:regU} and the conservation of energy, $\widehat E$ is controlled uniformly in time:
\be
\| \widehat E(t, \cdot) \|_{L^\infty(\TT^3)} \leq \exp{ \left ( C (1 + \| \rho_f(t, \cdot) \|_{L^{\frac{5}{3}}(\TT^3)}) \right )} \leq C(f_0).
\ee
The remaining step is to estimate $\bar E$. For this we use techniques established for the electron Vlasov-Poisson system by Pallard \cite{Pallard} and Chen and Chen \cite{Chen-Chen}; here we particularly follow the method of Chen and Chen \cite{Chen-Chen}.

First note that, by Lemma~\ref{lem:G}, there exists a constant $C$ such that for all $x \in \TT^d$,
\be \label{est:K-singularity}
|K(x)| \leq C (1 + |x|^{-2}).
\ee
It follows that
\be
|\bar E(t,x)| = |K \ast \rho_f(t,x) | \leq C \int_{x - \mc{Q}_3} \frac{\rho_f(t, y)}{|x-y|^2} \di y + C \| \rho_f\|_{L^1(\TT^d)}  .
\ee

To estimate the integral term we will use a technical lemma from \cite{Chen-Chen}. The estimate therein makes use of the fact that $f$ is the pushforward of its initial data along the characteristic flow of the electron Vlasov-Poisson system. In the VPME setting, the relevant characteristic flow has a different structure since $E = \bar E + \widehat E$, and so the estimate from \cite{Chen-Chen} cannot be applied immediately as stated. However, upon examining the proof it is possible to see that the estimate also applies to other characteristic flows, under the following set of assumptions.

\begin{hyp} \label{hyp:flow}
Let $X(t; s,x,v), V(t; s,x,v)$ denote a flow induced by a vector field of the form 
$$(v, a(t,x,v))$$ for some function $a$. That is,
	\be
	\begin{cases}
		\frac{\di}{\di t} X(t ; s,x,v) = V(t;s,x,v) \\
		\frac{\di}{\di t} V(t;s,x,v) = a(t, X(t;s,x,v) ,V(t;s,x,v)) \\
		(X(s;s,x,v), V(s;s,x,v)) = (x,v) .
	\end{cases}
	\ee
	Assume that the following properties hold.
	\begin{itemize}
		\item (Uniform control of small increments in velocity) Define $P(t,\delta)$ by
		\be
		P(t, \delta) : = \sup_{(x,v) \in \TT^3 \times \RR^3} \int_{t-\delta}^t |a(t, X(s;0,x,v) ,V(s;0,x,v))| \di s.
		\ee
		Assume that $P(t,\delta)$ is finite.
		\item Fix $f_0 \in L^1 \cap L^\infty(\TT^3 \times \RR^3)$ and let $f$ be the pushforward of $f_0$ along the flow.
		
		Assume that there exists a constant $C_0 > 0$ such that for all $t$
		\be
		\| f(t, \cdot, \cdot) \|_{L^\infty} \leq C_0, \qquad \int_{\TT^3 \times \RR^3} |v|^2 f(t,x,v) \di x \di v \leq C_0 .
		\ee
	\end{itemize}
\end{hyp}

By following the proof of \cite[Proposition 3.3]{Chen-Chen}, it is possible to see that the following estimate holds for any vector field satisfying Assumption~\ref{hyp:flow}.

\begin{lem}\label{lem:Q-controls-sing}
	Let $(X(t;s,x,v), V(t;s,x,v))$ and $f$ satisfy Assumption~\ref{hyp:flow}. Then
	\be
	\int_{t-\delta}^t \int_{X - \mc{Q}_3} \frac{\rho_f(s,y)}{|X(s; 0,x,v) - y|^2}  \di y  \di s \leq C \left ( \delta P(t,\delta)^{4/3} + \delta^{1/2} (1 + \delta P(t,\delta))^{1/2} P(t,\delta)^{-1/2} M_{3+\e}(t)^{1/2} \right ),
	\ee
	where the constant $C>0$ depends on $C_0,$ {and $\e>0$}.
\end{lem}

In particular, the flow induced by $(v, E(x))$, where $E$ is the electric field for a solution of the VPME system with finite energy, satisfies these assumptions. Henceforth we define the quantities $P(t,\delta)$ by
		\be
		P(t,\delta) : = \sup_{(x,v) \in \TT^3 \times \RR^3} \int_{t-\delta}^t |E(t, X(s;0,x,v)| \di s.
		\ee
Using the same proof as for \cite[Proposition 3.1]{Chen-Chen}, it can be shown that these quantities can be used to estimate the moments $M_{m}$.
\begin{lem}
Let $m>2$. There exists a constant $C>0$ depending on $M_m(0), M_{m-2}(0)$ and $M_2(t)$ such that
\be \label{Q-controls-moment}
M_m(t) \leq C \[ P(t,t)^{\max\{ 2, m-2\}} + 1 \].
\ee
\end{lem}

Using these results, we can now conclude the proof of Lemma~\ref{lem:moment-propagation-3d}.

\begin{proof}[Proof of Lemma~\ref{lem:moment-propagation-3d}]

We use Lemma~\ref{lem:Q-controls-sing} to bound $\bar E$. By \eqref{est:K-singularity}, along each trajectory $X^*(s) = X(s;0,x,v)$ we have
\begin{align} \label{est:barE-Q}
\int_{t-\delta}^t | \bar E(X^*(s))| \di s & \leq C \int_{t-\delta}^t \int_{X^*(s) - \mc{Q}_3} \frac{1}{|X^*(s) - y|^2} \rho_f(s,y) \di y  \di s + C \delta \| \rho_f \|_{L^1(\TT^d)} \\
& \leq C\delta + C \left ( \delta P(t,\delta)^{4/3} + \delta^{1/2} (1 + \delta P(t,\delta))^{1/2} P(t,\delta)^{-1/2} M_{3+\e}(t)^{1/2} \right ).
\end{align}

We then deduce a bound on $P(t,\delta)$ by using the decomposition $E = \bar E + \widehat E$. First, note that
\begin{align}
P(t,\delta) & = \sup_{x,v \in \TT^3 \times \RR^3} \int_{t-\delta}^t | E(X(s ; 0,x,v))| \di s \\
& \leq \sup_{x,v \in \TT^3 \times \RR^3} \left ( \int_{t-\delta}^t | \bar E(X(s ; 0,x,v))| \di s + \int_{t-\delta}^t | \widehat E(X(s ; 0,x,v))| \di s \right ) .
\end{align}
By Proposition~\ref{prop:regU}, $\widehat E$ is bounded, uniformly in $x$ and $t$. Thus there exists a constant $C>0$ depending only on $\| f_0 \|_{L^\infty}$ and $\mc{E}[f_0]$ such that
\be
P(t,\delta) \leq C \delta + \sup_{x,v \in \TT^3 \times \RR^3} \int_{t-\delta}^t | \bar E(X(s ; 0,x,v))| \di s .
\ee
Then, by the bound \eqref{est:barE-Q} on $\bar E$,
\be
P(t,\delta) \leq C \delta + C \left ( \delta P(t,\delta)^{4/3} + \delta^{1/2} (1 + \delta P(t,\delta))^{1/2} P(t,\delta)^{-1/2} M_{3+\e}(t)^{1/2} \right ) .
\ee
Multiplying by $P(t,\delta)^{1/2}$, we deduce that
\begin{align}
P(t,\delta)^{3/2} & \leq C \delta P(t,\delta)^{1/2} + C \left ( \delta P(t,\delta)^{11/6} + \delta^{1/2} (1 + \delta P(t,\delta))^{1/2} M_{3+\e}(t)^{1/2} \right ) \\
& \leq C \left ( \delta P(t,\delta)^{11/6} + \delta^{1/2} (1 + \delta P(t,\delta))^{1/2} (1+ M_{3+\e}(t)^{1/2}) \right ) .
\end{align}
Then, as explained in \cite[Proposition 3.3]{Chen-Chen}, it follows from the estimate above that
\be
P(t,t) \leq C (1 + t) \left ( 1 + \sup_{s \in [0,t]} M_{3+\e}(s) \right )^{\frac12} .
\ee
The end of the proof then follows as in \cite{Chen-Chen}. First, interpolate $M_{3+\e}$ between $M_2$ and $M_{m_0}$: by H\"{o}lder's inequality,
\be
M_{3+\e}(t) = \int_{\TT^3 \times \RR^3} |v|^{3 + \e} f \di x \di v \leq M_2(t)^{\frac{m_0-3-\e}{m_0-2}} M_{m_0}(t)^{\frac{1+\e}{m_0-2}} .
\ee
Thus
\be \label{Moment-controls-Q}
P(t,t) \leq C (1 + t)  \( 1 + \sup_{s \in [0,t]} M_{m_0}(s)^{\frac{1+\e}{2(m_0-2)}} \) .
\ee
Substituting estimate \eqref{Q-controls-moment} into \eqref{Moment-controls-Q}, we obtain
\be
P(t,t) \leq C (1 + t)  \( 1 + P(t,t)^{ \max{\left\{\frac{1+\e}{2}, \frac{1+\e}{m_0-2} \right\}}} \) .
\ee
For sufficiently small $\e$, the exponent $\max{\{\frac{1+\e}{2}, \frac{1+\e}{m_0-2} \}}$ is strictly less than one, and so there is a constant $C > 0$ depending on $t$, $\mc{E}[f_0]$, $\| f_0 \|_{L^\infty(\TT^3 \times \RR^3)}$ and $M_{m_0}(0)$ such that
\be
P(t,t) \leq C .
\ee
It then follows from \eqref{Q-controls-moment} that
\be
\sup_{t \in [0,T]} M_{m_0}(t) < + \infty,
\ee
as desired.
\end{proof}

\section{Construction of Solutions} \label{sec:construction}

In this section, we show global existence of weak solutions for the VPME system \eqref{vpme} for initial data with finite velocity moments of sufficiently high order.

\begin{thm} \label{thm:existence}
Let $d = 2, 3$. Consider an initial datum $f_0 \in L^1 \cap L^\infty(\bt^d \times \br^d)$ satisfying
\be
\int_{\TT^d \times \RR^d} |v|^{m_0} f_0(x,v) \di x \di v < + \infty, \; \;\mbox{for some}\,\, m_0 > d .
\ee
Then there exists a global-in-time weak solution $f \in C([0,\infty); \mc{P}(\bt^d \times \br^d))$ of the VPME system \eqref{vpme} with initial data $f_0$, such that for all $T>0$,
\be
\sup_{t \in [0,T]} \int_{\TT^d \times \RR^d} |v|^{m_0} f(t,x,v) \di x \di v < + \infty .
\ee
\end{thm}

To prove this theorem, we first consider a regularised system for which unique global solutions can be constructed. Then, using the a priori estimate from Section~\ref{sec:moments}, we extract a subsequential limit from the regularised solutions, and show that the limit is a weak solution of VPME \eqref{vpme} with the desired moments bounded.

\subsection{Regularised VPME}

We introduce a regularised version of \eqref{vpme}. We define a scaled mollifier $\chi_r$ by letting
\begin{align} \label{Def_chi}
\chi_r(x) = r^{-d} \chi \left ( \frac{x}{r} \right ) , \qquad r \in (0, \frac{1}{4} ]
\end{align}
Here $\chi: \bt^d \to \br$ is a fixed smooth function with support contained in the unit ball. We assume further that $\chi$ is radially symmetric, non-negative and has total mass 1. We then consider the following regularised system:
\begin{equation}
\label{vpme-reg}
 \left\{ \begin{array}{ccc}\pt_t f^{(r)}+v\cdot \nabla_x f^{(r)}+ E_r[f^{(r)}] \cdot \nabla_v f^{(r)} =0,  \\
E_r =- \chi_r \ast \nabla U_r, \\
\Delta U_r=e^{U_r} - \chi_r \ast \rho[f^{(r)}] ,\\
f^{(r)}\vert_{t=0}=f_0 \ge0,\ \  \int_{\bt^d \times \br^d} f_0 \,dx\,dv=1.
\end{array} \right.
\end{equation}

We regularise the ion density but not the electron density, the idea being that the thermalisation assumption should lead to a regularising effect. This is a slightly different approach from that of Bouchut \cite{Bouchut}, where both densities are regularised.

We introduce the decomposition
\be
E_r = \bar E_r + \widehat E_r,
\ee
where
\be
\bar E_r = - \chi_r \ast \nabla \bar U_r, \qquad \widehat E_r = - \chi_r \ast \nabla \widehat U_r,
\ee
with $\bar U_r, \, \widehat U_r$ satisfying
\be
\Delta \bar U_r = 1 - \chi_r \ast \rho[f^{(r)}], \qquad \Delta \widehat U_r = e^{\bar U_r + \widehat U_r} - 1 .
\ee

Notice that we are using a technique of `double regularisation'; for instance, 
$\bar E_r$ can be represented in the form
\be
\bar E_r = \chi_r \ast \chi_r \ast K \ast \rho[f^{(r)}] .
\ee
This type of regularisation appeared in the work of Horst \cite{Horst}, and has subsequently been used in many other contexts. An advantage of this approach is that the system \eqref{vpme-reg} has an associated conserved energy, defined by
\be \label{def:energy-reg}
\mc{E}_r [f] : = \frac{1}{2} \int_{\bt^d \times \br^d} |v|^2 f \di x \di v + \frac{1}{2} \int_{\bt^d} \lvert \nabla U_r \rvert^2 \di x + \int_{\bt^d} U_r e^{U_r} \di x .
\ee
If $f^{(r)}$ converges to some $f$ sufficiently strongly as $r$ tends to zero, then we would expect $\mc{E}_r [f^{(r)}]$ to converge to $\mc{E}[f]$, where $\mc{E}$ is the energy of the original VPME system, defined in \eqref{def:Ee}.

The methods of Dobrushin \cite{Dob} may be used to construct solutions to this regularised system since the force-field is sufficiently regular. Dobrushin's results cannot be applied directly since the force is not of convolution type, but the method can be adapted to our case.

\begin{lem}[Existence of regularised solutions] \label{lem:exist-vpme-reg}
For every $f_0 \in \mc{P}(\bt^d \times \br^d)$, there exists a unique solution $f^{(r)} \in C([0,\infty) ; \mc{P}(\bt^d \times \br^d))$ of \eqref{vpme-reg}. If $f_0 \in L^p(\bt^d \times \br^d)$ for some $p \in [1, \infty]$, then for all $t$
\be
\lVert f^{(r)}(t) \rVert_{L^p(\bt^d \times \br^d)} \leq \lVert f_0 \rVert_{L^p(\bt^d \times \br^d)} .
\ee
\end{lem}

\begin{proof}
We sketch the proof, which is a modification of the methods of \cite{Dob} in order to handle the extra term in the electric field.
First consider the linear problem for fixed $\mu \in C([0,\infty) ; \mc{P}(\bt^d \times \br^d))$:
\begin{equation}
\label{lin-vpme-reg}
 \left\{ \begin{array}{ccc}\pt_t g^{(\mu)}_r+v\cdot \nabla_x g^{(\mu)}_r+ E_r[\mu] \cdot \nabla_v g^{(\mu)}_r=0,  \\
E^{(\mu)}_r =- \chi_r \ast \nabla U^{(\mu)}_r, \\
\Delta U^{(\mu)}_r=e^{U^{(\mu)}_r} - \chi_r \ast \rho[\mu] ,\\
g^{(\mu)}_r\vert_{t=0}=f_0\ge0,\ \  \int_{\bt^d \times \br^d} f_0(\di x \di v)=1,
\end{array} \right.
\end{equation}
for $f_0 \in \mc{P}(\bt^d \times \br^d)$. Observe that even when $\mu$ is a singular probability measure, $\chi_r \ast \rho[\mu]$ is a function satisfying 
\be \label{rho-mu-conv-linfty}
| \chi_r \ast \rho[\mu]| \leq \lVert \chi_r \rVert_{L^{\infty}(\bt^d)} .
\ee
Then by Proposition~\ref{prop:regU},
\be
\| U^{(\mu)}_r \|_{C^1(\bt^d)} \leq \exp{\left [ C_d\,\(1 + \lVert \chi_r \rVert_{L^{\infty}(\bt^d)}\) \right ]} ,
\ee
and hence $E^{(\mu)}_r = \chi_r \ast \nabla U^{(\mu)}_r$ is of class $C^1(\bt^d)$, with the uniform-in-time estimate
\be \label{nabU-reg}
\lVert E^{(\mu)}_r \rVert_{C^1(\bt^d)} \leq \lVert \chi_r \rVert_{C^1(\bt^d)}
\lVert \nabla U^{(\mu)}_r \rVert_{C(\bt^d)} \leq \lVert \chi_r \rVert_{C^1(\bt^d)} \exp{\left [ C_d \(1 + \lVert \chi_r \rVert_{L^{\infty}(\bt^d)}\) \right ]} \leq C_{r,d}\, .
\ee
This implies the existence of a unique global-in-time $C^1$ characteristic flow. Using this flow we may construct a unique solution $g^{(\mu)}_r \in C([0,\infty) ; \mc{P}(\bt^d \times \br^d))$ to the linear problem \eqref{lin-vpme-reg} by the method of characteristics. Since the vector field $(v, E_r)$ is divergence free, this solution conserves $L^p(\bt^d \times \br^d)$ norms for $p \in [1, \infty]$.

To prove existence for the nonlinear equation, we use a fixed point argument via a contraction estimate in Wasserstein sense, as in \cite{Dob}. To prove the required contraction estimate, it is enough to show that the electric field $E^{(\mu)}_r$ is Lipschitz and has a stability property in $W_1$ with respect to $\mu$:
\begin{align} \label{Er-reg}
\| E^{(\mu)}_r \|_{\text{Lip}} & \leq C_r \\ \label{Er-stab}
\| E^{(\mu)}_r - E^{(\nu)}_r\|_{L^\infty(\bt^d)} & \leq C_r W_1(\mu, \nu) .
\end{align}
The Lipschitz regularity \eqref{Er-reg} holds by \eqref{nabU-reg}. For the stability \eqref{Er-stab}, once again we use the decomposition $E^{(\mu)}_r = \bar E^{(\mu)}_r + \widehat E^{(\mu)}_r$. First, 
$$
\bar E^{(\mu)}_r = - \chi_r \ast \nabla \bar U^{(\mu)}_r = \chi_r \ast K \ast \chi_r \ast \rho[\mu],
$$
where $K$ is the Coulomb kernel as defined by $K = \nabla G$ for $G$ satisfying \eqref{def:G}.
This is a force of convolution type, with a Lipschitz kernel since $K \in L^1(\bt^d)$ and $\chi_r$ is smooth, so the required stability estimate is proved in \cite{Dob}. It remains to verify stability of $\widehat E_r$ with respect to $\mu$. 

\noindent Consider two continuous paths of probability measures $\mu, \nu \in C([0,\infty) ; \mc{P}(\bt^d \times \br^d))$. First note that by Young's inequality,
\be
\| \widehat E^{(\mu)}_r - \widehat E^{(\nu)}_r \|_{L^\infty(\bt^d)}  = \lVert \chi_r \ast ( \nabla \widehat U^{(\mu)}_r - \nabla \widehat U^{(\nu)}_r ) \rVert_{L^{\infty}(\bt^d)}  \leq \lVert \chi_r \rVert_{L^2(\bt^d)} \lVert \nabla \widehat U^{(\mu)}_r - \nabla \widehat U^{(\nu)}_r \rVert_{L^{2}(\bt^d)} .
\ee
By the $L^2$ stability estimate from Lemma~\ref{lem:hatU-stab}, 
\be 
\lVert \nabla \widehat U^{(\mu)}_r - \nabla \widehat U^{(\nu)}_r \rVert_{L^{2}(\bt^d)} \leq \exp{\left [ C \left ( \max_{\gamma \in \{\mu, \nu\}}\, \lVert \bar{U}^{(\gamma)}_r\rVert_{L^{\infty}(\bt^d)}  + \max_{\gamma \in \{\mu, \nu\}} \, \lVert \widehat{U}^{(\gamma)}_r\rVert_{L^{\infty}(\bt^d)}   \right )\right ]} \, \lVert  \bar U^{(\mu)}_r -  \bar U^{(\nu)}_r \rVert_{L^{2}(\bt^d)} .
\ee
By Proposition~\ref{prop:regU},
\be
 \max_{\gamma \in \{\mu, \nu\}}\,\lVert \bar{U}^{(\gamma)}_r\rVert_{L^{\infty}(\bt^d)}  + \max_{\gamma \in \{\mu, \nu\}}\, \lVert \widehat{U}^{(\gamma)}_r\rVert_{L^{\infty}(\bt^d)} \leq  \exp{\left [ C_d \left (1 + \lVert \chi_r \rVert_{L^{\infty}(\bt^d)} \right ) \right ]} .
\ee
Hence
\begin{align}
\lVert \nabla \widehat U^{(\mu)}_r - \nabla \widehat U^{(\nu)}_r \rVert_{L^{2}(\bt^d)} & \leq C_{r,d}\,\,  \lVert  \bar U^{(\mu)}_r -  \bar U^{(\nu)}_r \rVert_{L^{2}(\bt^d)} \\
& \leq C_{r,d}\,\,  \lVert  \bar U^{(\mu)}_r -  \bar U^{(\nu)}_r \rVert_{L^{\infty}(\bt^d)}  =  C_{r,d}\,\,  \lVert \chi_r \ast_x G \ast_x \(\rho[\mu] - \rho[\nu]\) \rVert_{L^{\infty}(\bt^d)}   .
\end{align}
Note that $\chi_r \ast_x G$ is smooth and hence Lipschitz. By Kantorovich duality for the $W_1$ distance we have
\be
W_1(\rho_\mu, \rho_\nu) = \sup_{\|\phi\|_{\text{Lip}} \leq 1} \left \{ \int_{\bt^d} \phi \di \rho_\mu - \int_{\bt^d} \phi \di \rho_\nu \right \} .
\ee
Thus for any $x \in \mathbb T^d$
\begin{align}
\chi_r \ast_x G \ast_x (\rho_\mu - \rho_\nu)(x)
&=\int_{\mathbb T^d}[\chi_r \ast_x G]  (x-y)\, d(\rho_\mu - \rho_\nu)(y)\\
&\leq \|\chi_r \ast_x G  (x-\cdot)\|_{{\rm Lip}} W_1(\rho_\mu, \rho_\nu)\leq C_{r,d}\, W_1(\rho_\mu, \rho_\nu),
\end{align}
where $C_{r,d}$ is independent of $x$.
Hence
$$
 \lVert \chi_r \ast_x G \ast_x (\rho_\mu - \rho_\nu) \rVert_{L^{\infty}(\bt^d)} \leq C_{r,d}\, W_1(\rho_\mu, \rho_\nu) .
$$
We conclude that
\begin{align}
\lVert \chi_r \ast ( \nabla \widehat U^{(\mu)}_r - \nabla \widehat U^{(\nu)}_r ) \rVert_{L^{\infty}(\bt^d)} \leq C_{r,d}\, W_1(\rho_\mu, \rho_\nu)   \leq C_{r,d}\, W_1(\mu, \nu)    ,
\end{align}
which shows that \eqref{Er-stab} holds.

Using the methods of \cite{Dob}, we can show that the estimates \eqref{Er-reg} and \eqref{Er-stab} imply a Wasserstein stability estimate:
\be
W_1\(g^{(\mu)}_r(t), g^{(\nu)}_r(t)\) \leq \int_0^t W_1(\mu(t), \nu(t)) \exp(C_r (t-s)) \di s .
\ee
Since $C_r$ is independent of time, a Picard iteration proves the existence of a unique solution $f^{(r)} \in C([0, \infty) ; \mc{P}(\bt^d \times \br^d))$ for the nonlinear regularised equation \eqref{vpme-reg}.
 {This solution also preserves all $L^p(\bt^d \times \br^d)$ norms, since it is the classical solution of a linear transport equation 
with divergence-free vector field $(v, E_r[f^{(r)}])$.}

\end{proof}

\subsection{Compactness}

Next, we show that the approximate solutions $f^{(r)}$ converge to a limit as $r$ tends to zero, and that this limit may be identified as the unique bounded density solution of \eqref{vpme} with data $f_0$. In the following lemma, we collect together some useful uniform estimates for the approximate solutions $f^{(r)}$.

\begin{lem} \label{lem:fr-bounds}
Let $f_0 \in L^1\cap L^\infty(\bt^d \times \br^d)$ have a finite velocity moment of order $m_0 > d$, that is, $M_{m_0}(0) < +\infty$. For each $r > 0$, let $f^{(r)}$ denote the solution of \eqref{vpme-reg} with initial datum $f_0$. Then $f^{(r)}$ have the following properties:
\begin{enumerate}[(i)]
\item $L^p$ bounds: for all $p \in [1, \infty]$,
\be \label{vpme-reg-Lp}
\sup_r \sup_{t \in [0,T]} \| f^{(r)}(t) \|_{L^p(\bt^d \times \br^d)} \leq \| f_0 \|_{L^p(\bt^d \times \br^d)} .
\ee
\item Moment bounds:
\begin{align} \label{f-reg-mom}
&\sup_r \sup_{t \in [0,T]} \int_{\bt^d \times \br^d} |v|^2 f^{(r)}(t,x,v) \di x \di v \leq C(T,f_0),\\
& \sup_r \sup_{t \in [0,T]} \int_{\bt^d \times \br^d} |v|^{m_0} f^{(r)}(t,x,v) \di x \di v \leq C(T,f_0) . 
\end{align}
\item Density bounds: for all $r$ and all $t \in [0,T]$,
\begin{align}
\label{rho-reg-bdd}
\sup_r \| \rho[f^{(r)}](t, \cdot) \|_{L^{\frac{d+2}{d}}(\bt^d)} \leq C(f_0), \qquad
\sup_r \| \rho[f^{(r)}](t, \cdot) \|_{ L^ 2(\bt^d)} \leq C(T, f_0) .
\end{align}
\item Regularity of the electric field: for any $\alpha \in (0,1)$,

\begin{align} \label{elec-reg}
&\sup_{r} \sup_{t \in [0,T]} \| \widehat U_r(t) \|_{C^{1, \alpha}(\bt^d)} \leq C(\alpha, f_0),\\
& \sup_{r} \sup_{t \in [0,T]} \| \bar U_r(t) \|_{C^{0, \alpha}(\bt^d)} \leq C(\alpha, f_0), \\
&\sup_{r} \sup_{t \in [0,T]} \| \bar U_r(t) \|_{W^{2,2}(\bt^d)} \leq C(T, f_0) .
\end{align}

\item Equicontinuity in time into $W^{-1, 2}$: for any $t_1<t_2,$
\be \label{vpme-reg-equi}
\|f^{(r)}({t_2})-f^{(r)}({t_1})\|_{W^{-1,2}(\mathbb T^d\times \mathbb R^d)} \le C(f_0) \, |t_2-t_1|,
\ee
where $W^{-1,2}(\mathbb T^d\times \mathbb R^d)$ denotes the dual of $W^{1,2}(\mathbb T^d\times \mathbb R^d)$. 
\end{enumerate}
\end{lem}

\begin{proof}
Property (i) was proved in Lemma~\ref{lem:exist-vpme-reg}. For Property (ii), the second moment bound is a consequence of the conservation of the energy functional $\mc{E}_r [f^{(r)}]$ defined by \eqref{def:energy-reg}, once we check that $\mc{E}_r[f_0]$ is bounded uniformly in $r$.

Since $f_0$ is in $L^\infty(\bt^d \times \br^d)$ and has a finite $m_0$-th order moment in velocity for some $m_0>d$, the second moment of $f_0$ is also finite. Then, by Lemma~\ref{lem:mom-interpolation} it follows that $\rho_{f_0} \in L^{\frac{d+2}{d}}(\TT^d)$. Since $\| \chi_r \ast \rho_{f_0} \|_{L^{\frac{d+2}{d}}(\bt^d)} \leq \|\rho_{f_0} \|_{L^{\frac{d+2}{d}}(\bt^d)}$, 
by Proposition~\ref{prop:regU} we have
\be
\| U_r(0) \|_{L^\infty(\bt^d)} \leq C\(\|\rho_{f_0}\|_{L^{\frac{d+2}{d}}(\bt^d)}\) ,
\ee
and so
\be
\| \Delta U_r(0) \|_{L^{\frac{d+2}{d}}(\TT^d)} \leq C\(\|\rho_{f_0}\|_{L^{\frac{d+2}{d}}(\bt^d)}\) .
\ee
By regularity estimates for the Poisson equation,
\be
\| U_r(0) \|_{W^{2,\frac{d+2}{d}}(\TT^d)} \leq C\(\|\rho_{f_0}\|_{L^{\frac{d+2}{d}}(\bt^d)}\) .
\ee
A Sobolev inequality then implies that
\be
\| \nabla U_r(0) \|_{L^{p(d)}(\TT^d)} \leq C\(\|\rho_{f_0}\|_{L^{\frac{d+2}{d}}(\bt^d)}\),
\ee
where $p(2)$ may be any $p < +\infty$, and $p(3) = \frac{15}{4}$. In either case it follows that
\be
\| \nabla U_r(0) \|_{L^{p(d)}(\TT^d)} \leq C\(\|\rho_{f_0}\|_{L^{\frac{d+2}{d}}(\bt^d)}\) .
\ee
Altogether it follows that there exists $C_0$ such that
\be \label{vpmereg-energy-initial}
\sup_{r} \sup_{t \in [0,T]} \mc{E}_r[f_0] \leq C_0,
\ee
therefore
\be \label{vpmereg-energy}
\sup_{r} \sup_{t \in [0,T]} \mc{E}_r[f^{(r)}(t)] \leq C_0.
\ee
Note that $x e^x \geq -e^{-1}$. This implies that for all $r$ and all $t \in [0,T]$,
\be \label{vpmereg-mom2}
\int_{\bt^d \times \br^d} |v|^2 f^{(r)}(t) \di x \di v \leq C,
\ee
which completes the proof of the first part of Property (ii).

The estimate on the moment of order $m_0$ is proven using the same arguments as used for Proposition~\ref{prop:moment-propagation}.
For the two dimensional case, since $\| E_r \|_{L^p(\TT^d)} \leq \| \nabla U_r \|_{L^p(\TT^d)}$ and $\| \chi_r \ast \rho [f^{(r)}] \|_{L^p(\TT^d)} \leq \|\rho[f^{(r)}] \|_{L^p(\TT^d)}$ for any $p \in [1, +\infty]$, the same proof follows through. For the three dimensional case, instead of the quantities $P(t,\delta)$ consider
\be
Q^{(r)}(t,\delta) : = \sup_{x,v \in \TT^3 \times \RR^3} \int_{t-\delta}^t |E_r (X_r(s; 0,x,v))| \di s,
\ee
where $(X_r(t; s,x,v), V_r(t;s,x,v))$ denotes the characteristic flow induced by the vector field $(v, E_r(x))$. $Q^{(r)}$ is well-defined due to estimate \eqref{nabU-reg} which shows that $E_r \in L^\infty([0, +\infty) \times \TT^3)$, and so $Q^{(r)}(t,\delta) \leq C_{r} \delta$ for some constant $C_{r}$. Thus observe that the flow $(X_r(t; s,x,v), V_r(t;s,x,v))$ satisfies Assumption~\ref{hyp:flow}. Therefore, the conclusion of Lemma~\ref{lem:Q-controls-sing} applies for this flow:
	\begin{multline}
	\int_{t-\delta}^t \int_{\mc{Q}_3} \frac{1}{|X_r(s; 0,x,v) - y|^2} \rho_{f^{(r)}}(s,y) \di y  \di s \\
	\leq C \left ( \delta Q^{(r)}(t,\delta)^{4/3} + \delta^{1/2} (1 + \delta Q^{(r)}(t,\delta))^{1/2} Q^{(r)}(t,\delta)^{-1/2} M_{3+\e}(t)^{1/2} \right ).
	\end{multline}
	To bound $E_r$, we use the decomposition
\be
|E_r(X_r(s))| \leq C \int_{\mc{Q}_3}|K_r \ast \rho_{f^{(r)}}(s,y)| \di y + |K_0 \ast \rho_{f^{(r)}}(X_r(s))|  +  |\widehat E_r(X_r(s))|,
\ee
where $K_r$ denotes the following regularisation of the singular part of the kernel:
\be
K_r = \chi_r \ast \chi_r \ast \frac{1}{|\cdot |^2} .
\ee
Proposition~\ref{prop:regU} implies that $\widehat E_r$ is uniformly bounded due to the conservation of energy.	
In Lemma~\ref{lem:reg-kernel} below, we show that $K_r$ is controlled by the unregularised kernel $|x|^{-2}$. Thus
\be
|E_r(X_r(s))| \leq C \int_{\mc{Q}_3} \frac{1}{|X_r(s; 0,x,v) - y|^2} \rho_{f^{(r)}}(s,y) \di y +  C ,
\ee
and so
\be
Q^{(r)}(t,\delta) \leq C \left ( \delta Q^{(r)}(t,\delta)^{4/3} + \delta^{1/2} (1 + \delta Q^{(r)}(t,\delta))^{1/2} Q^{(r)}(t,\delta)^{-1/2} M_{3+\e}(t)^{1/2} \right ).
\ee
By following the remainder of the proof of Lemma~\ref{lem:moment-propagation-3d}, we can deduce a bound on the velocity moment of order $m$ of $f^{(r)}$, which completes the proof of Property (ii).

Property (iii) then follows from Property (ii) after applying Lemma~\ref{lem:mom-interpolation} with $k=0$ and either $m=2$ or $m=d$.

For Property (iv), we first note that $\| \chi_r \ast \rho[f(t)] \|_{L^{p}(\bt^d)} \leq \|\rho [f(t)] \|_{L^{p}(\bt^d)}$ for any $p \in [1, + \infty]$. The $C^{1,\alpha}$ estimate on $\widehat U_r$ then follows directly from Proposition~\ref{prop:regU} and the $L^{\frac{d+2}{d}}$ estimate from Property (iii). The $W^{2,2}$ estimate on $\bar U_r$ follows from regularity properties of the Poisson equation and the $L^2$ estimate on $\rho_{f^{(r)}}$ from Property (iii).

Finally, we consider the equicontinuity in time. By the transport equation
$$
\partial_t f^{(r)} =-{\rm div}_{x,v}\(\bigl(v,E_r[{f^{(r)}(t)}]\bigr) f^{(r)}\)
$$
and the bounds \eqref{vpme-reg-Lp} and \eqref{vpmereg-energy}, for any function $\phi\in W^{1,2}(\mathbb T^d\times \mathbb R^d)$,
\begin{align*}
\biggl|\frac{d}{dt}\int_{\bt^d \times \br^d} \phi f^{(r)}(t)\,dx&\,dv\biggr|=\biggl|\int_{\bt^d \times \br^d} \nabla_{x,v} \phi\cdot \bigl(v,E_r[{f^{(r)}(t)}]\bigr) f^{(r)}(t)\,dx\,dv\biggr|
\\&\leq \biggl(\int_{\bt^d \times \br^d} |\nabla_{x,v} \phi|^2 f^{(r)}(t) \di x \di v \biggr)^{1/2} \biggl(\int_{\bt^d \times \br^d}(|v|^2+ |E_r[{f^{(r)}(t)}]|^2)f^{(r)}(t) \di x \di v\biggr)^{1/2}\\
&\leq \biggl(\int_{\bt^d \times \br^d} |\nabla_{x,v} \phi|^2 \di x \di v \biggr)^{1/2}\lVert f^{(r)}(t) \rVert_{L^{\infty}(\bt^d \times \br^d)}^{1/2} \\
& \quad \times \biggl(\int_{\bt^d \times \br^d} |v|^2f^{(r)}(t) \di x \di v +\lVert f^{(r)}(t) \rVert_{L^{\infty}(\bt^d \times \br^d)}\int_{\bt^d \times \br^d} |E_r[{f^{(r)}(t)}]|^2 \di x \di v \biggr)^{1/2} \\
&\leq \biggl(\int_{\bt^d \times \br^d} |\nabla_{x,v} \phi|^2 \di x \di v \biggr)^{1/2}\lVert f^{(r)}(t) \rVert_{L^{\infty}(\bt^d \times \br^d)} \\
& \quad \times \biggl(\int_{\bt^d \times \br^d} |v|^2f^{(r)}(t) \di x \di v + \int_{\bt^d \times \br^d} |\nabla U_r|^2 \di x \di v \biggr)^{1/2} \\
& \leq C \lVert f_0 \rVert_{L^{\infty}(\bt^d \times \br^d)} \, \( C + \mc{E}_r[f^{(r)}(t)]\) \, \|  \nabla \phi\|_{L^2(\mathbb T^d\times \mathbb R^d)} \\
&\leq C[f_0] \|\nabla_{x,v} \phi\|_{L^2(\mathbb T^d\times \mathbb R^d)}.
\end{align*}
This estimate means that 
$$
\|\pt_t f^{(r)}(t)\|_{W^{-1,2}(\mathbb T^d\times \mathbb R^d)}= \sup_{\|\phi\|_{W^{1,2}(\mathbb T^d\times \mathbb R^d)}\leq 1}\int_{\bt^d \times \br^d} \phi\, \partial_tf^{(r)}(t)\,dx\,dv \leq C,
$$
thus $\partial_t f^{(r)} \in L^\infty((0,T);W^{-1,2}(\mathbb T^d\times \mathbb R^d))$. 
Thus, for any $t_1<t_2,$
\be
\|f^{(r)}(t_2)-f^{(r)}(t_1)\|_{W^{-1,2}(\mathbb T^d\times \mathbb R^d)}\le \int_{t_1}^{t_2}\|\pt_t f^{(r)}(t)\|_{W^{-1,2}(\mathbb T^d\times \mathbb R^d)}\,dt\le C|t_2-t_1|,
\ee
which completes the proof of Property (v).
\end{proof}

The following lemma, used in the proof above, shows that the regularised kernel $K_r$ can be controlled by the function $|x|^{-(d-1)}$ near the singularity, uniformly in $r$.

\begin{lem}[Bounds on the regularised kernel] \label{lem:reg-kernel}
Let $d > 1$ and let $\chi_r$ be defined by \eqref{Def_chi} for some fixed $\chi \in C(\bt^d)$.
Then there exists a constant $C(d, \chi) > 0$, independent of $r$, such that, for all $x \in \bt^d$,
\be
| \chi_r \ast K (x) | \leq C(d, \chi) \left ( 1 + |x|^{-(d-1)} \right ) .
\ee

\end{lem}
\begin{proof}

By Lemma~\ref{lem:G}, there exists a constant $C_d$ such that for all $x\in \TT^d$,
\be
|K(x)| \leq C_d (1 + |x|^{-(d-1)}) ,
\ee
where $|\cdot|$ denotes the distance on the torus defined by \eqref{def:torus-metric}.
Therefore
\be
|\chi_r \ast K(x)| \leq C_d \int_{\mc{Q}_d}  \left ( 1 + |x-y|^{-(d-1)} \right ) \chi_r(y) \di y .
\ee

First, observe that
\begin{align}
|\chi_r \ast K(x)| &\leq C_d \| \chi_r \|_{L^1(\TT^d)} + C_d \int_{\mc{Q}_d}  |x-y|^{-(d-1)} \chi_r(y) \di y \\
& \leq C_d \| \chi \|_{L^1(\TT^d)} + C_d \int_{\mc{Q}_d}  |x-y|^{-(d-1)} \chi_r(y) \di y \\
& \leq C(d, \chi) + C_d \int_{\mc{Q}_d}  |x-y|^{-(d-1)} \chi_r(y) \di y .
\end{align}

We then consider the function
\be
|x|^{d-1} \int_{\mc{Q}_d} \frac{\chi_r(y)}{|x-y|^{d-1}} \di y .
\ee
There exists a constant $C_d > 0$ such that
\be
|x|^{d-1} \leq C_d \left ( |x-y|^{d-1} + |y|^{d-1} \right ) .
\ee
Thus
\be \label{reg-kernel-decomp}
|x|^{d-1} \int_{\mc{Q}_d} \frac{\chi_r(y)}{|x-y|^{d-1}} \di y \leq \int_{\mc{Q}_d} | \chi_r(y) | \di y + \int_{\mc{Q}_d} \frac{|y|^{d-1}}{|x-y|^{d-1}} |\chi_r(y)| \di y.
\ee
Note that, for $r \leq \frac{1}{2}$,
\be
\int_{\mc{Q}_d} | \chi_r(y) | \di y = r^{-d} \int_{\mc{Q}_d} \chi \left ( \frac{y}{r} \right ) \di y = \int_{B_1(0)} \chi (y) \di y = C(\chi) .
\ee
Split the second term of \eqref{reg-kernel-decomp} as follows: for any $L \in (0, \frac{1}{2} \sqrt{d} ]$,
\be
\int_{\mc{Q}_d} \frac{|y|^{d-1}}{|x-y|^{d-1}} |\chi_r(y)| \di y \leq \int_{y \in \mc{Q}_d : |x-y| \leq L} \frac{|y|^{d-1}}{|x-y|^{d-1}} |\chi_r(y)| \di y +\int_{y \in \mc{Q}_d : |x-y| > L} \frac{|y|^{d-1}}{|x-y|^{d-1}} |\chi_r(y)| \di y .
\ee
The first term is estimated by
\be
\int_{|x-y| \leq L} \frac{|y|^{d-1}}{|x-y|^{d-1}} |\chi_r(y)| \di y \leq \| |\cdot|^{d-1} \chi_r \|_{L^\infty(\bt^d)} \int_{|y| \leq L} \frac{1}{|y|^{d-1}} \di y \leq C_d L \| |\cdot|^{d-1} \chi_r \|_{L^\infty(\bt^d)}  .
\ee
Observe that, for $x \in \mc{Q}_d$ and $r \leq \frac{1}{2}$,
\be
|x|^{d-1} | \chi_r (x) | = r^{-1}  \left | \frac{x}{r} \right |^{d-1} \chi \left ( \frac{x}{r} \right )  \leq r^{-1}  \| |\cdot|^{d-1} \chi \|_{L^\infty(B_1(0))}  \leq C(\chi) r^{-1}.
\ee
The second term is estimated by
\be
\int_{y \in \mc{Q}_d : |x-y| > L} \frac{|y|^{d-1}}{|x-y|^{d-1}} |\chi_r(y)| \di y \leq L^{1-d} \| |\cdot|^{d-1} \chi_r \|_{L^1(\bt^d)} .
\ee
For the constant, we find that
\begin{align}
\| |\cdot|^{d-1} \chi_r \|_{L^1(\bt^d)} &= r^{-d} \int_{\mc{Q}_d} |x|^{d-1} \left | \chi  \left ( \frac{x}{r} \right )  \right | \di x  = r^{d-1} \int_{B_1(0)} |x|^{d-1} \chi (x) \di x  \leq C(\chi)\,  r^{d-1} .
\end{align}
Altogether this gives
\be
|x|^{d-1} \int_{\mc{Q}_d} \frac{\chi_r(y)}{|x-y|^{d-1}} \di y \leq C(\chi) \[ 1 + r^{-1} \left ( C_d L + L^{1-d} r^d \right )\] .
\ee
Minimising over $L$, the optimal value is $L = C_d r$.
Then
\be
|x|^{d-1} \int_{\mc{Q}_d} \frac{\chi_r(y)}{|x-y|^{d-1}} \di y \leq C(d, \chi).
\ee
This completes the proof.
\end{proof}

\noindent In the next lemma, we use the above bounds to extract a convergent subsequence of approximate solutions, and show that the limit is a weak solution of \eqref{vpme}. This completes the proof of Theorem~\ref{thm:existence}.

\begin{lem} \label{lem:exists-weak-solution}
Let $f_0 \in L^1\cap L^\infty(\bt^d \times \br^d)$ be compactly supported. For each $r > 0$, let $f^{(r)}$ denote the solution of \eqref{vpme-reg} with initial datum $f_0$. Then there exists a subsequence $f^{(r_n)}$ converging to a limit $f \in C([0,\infty) ; W^{-1,2}(\bt^d \times \br^d))$, in the sense that for each time horizon $T > 0$ and for all $\phi \in W^{1,2}(\bt^d \times \br^d)$,
\be
\lim_{n \to \infty} \sup_{t \in [0,T]} \left \lvert \int_{\bt^d \times \br^d} \phi \(f^{(r_n)}(t) - f(t)\)\di x \di v \right \rvert = 0 .
\ee
Moreover, for each $t \in [0,\infty)$, for any $p \in [1,\infty]$ and all $\phi \in L^p(\bt^d \times \br^d)$,
\be
\lim_{n \to \infty} \left \lvert \int_{\bt^d \times \br^d} \phi \(f^{(r_n)}(t) - f(t) \)\di x \di v \right \rvert = 0 .
\ee
Furthermore $f$ is a weak solution of \eqref{vpme} with initial datum $f_0$, for which
\be
\sup_{t \in [0,T]} \int_{\TT^d \times \RR^d} |v|^m_0 f(t,x,v) \di x \di v < + \infty \quad \text{ for all } T>0
\ee
and
\be
 \mc{E}[f(t)] =  \mc{E}[f_0] \qquad \text{ for all } t \in [0, +\infty).
\ee

\end{lem}
\begin{proof}
To extract the convergent subsequence, we need to make careful use of the equicontinuity in time. The curves
$$
t\mapsto f^{(r)}(t) \in W^{-1,2}(\mathbb T^d\times \mathbb R^d)
$$
are equicontinuous in the norm topology on $W^{-1,2}(\bt^d \times \br^d)$ by \eqref{vpme-reg-equi}. They are also uniformly bounded in $W^{-1,2}(\bt^d \times \br^d)$ since $f^{(r)} \in L^\infty([0,+\infty) ; L^2(\bt^d \times \br^d))$ by \eqref{vpme-reg-Lp}. 
By an Arzel\`{a}-Ascoli type argument we may extract a subsequence $r_n$ such that for all $\phi \in W^{1,2}(\bt^d \times \br^d)$ and all $T > 0$,
\be \label{f-lim}
\lim_{n \to \infty} \sup_{t \in [0,T]} \left \lvert \int_{\bt^d \times \br^d}\(f^{(r_n)}(t) - f(t)\) \phi \di x \di v \right \rvert = 0 ,
\ee
for some $f \in C([0,+\infty) ; W^{-1, 2}(\bt^d \times \br^d))$. In particular, since $C^\infty_c(\mathbb T^d\times \mathbb R^d)\subset W^{1,2}(\mathbb T^d\times \mathbb R^d)$,
\begin{equation}
\label{eq:lim W12}
\lim_{n \to \infty} \sup_{t \in [0,T]} \left \lvert \int_{\bt^d \times \br^d}\(f^{(r_n)}(t) - f(t)\) \phi \di x \di v \right \rvert = 0 , \quad \text{for all}\, \, \, \phi \in C_c^\infty(\mathbb T^d\times \mathbb R^d), \, T>0.
\end{equation}
We now want to prove that the convergence also holds weakly in $L^p(\bt^d \times \br^d)$, for $p \in [1, \infty)$, and in $L^\infty(\bt^d \times \br^d)$ in weak$^*$ sense.
For each fixed $t$, we have the uniform bounds
\begin{align} \label{vpme-reg-Lp-2}
\sup_r \| f^{(r)}(t) \|_{L^p(\bt^d \times \br^d)} \leq \| f_0 \|_{L^p(\bt^d \times \br^d)}, \qquad
\sup_r  \int_{\bt^d \times \br^d}  |v|^2 f^{(r)}(t,x,v) \di x \di v \leq C(f_0) .
\end{align}
This implies that $\{ f^{(r)}(t)  \}_{r>0}$ is relatively compact in $L^p(\bt^d \times \br^d)$ with respect to the weak topology for $p \in [1,\infty)$ and in $L^\infty(\bt^d \times \br^d)$ with respect to the weak$^*$ topology. For each $p \in [1,\infty]$ and $t$ there is a further subsequence $r_{n_k}$ and a limit $g \in L^{p}(\bt^d \times \br^d)$, both depending on $t$ and $p$, such that for all $\phi \in L^{p^*}(\bt^d \times \br^d)$ ($p^*$ being the H\"older conjugate of $p$),
\be
\lim_{k \to \infty} \left \lvert \int_{\bt^d \times \br^d} \phi (f^{(r_{n_k})}(t) - g)\di x \di v \right \rvert = 0 .
\ee
In particular, this holds for $\phi \in C^\infty_c(\mathbb T^d\times \mathbb R^d)\subset L^{p^*}(\mathbb T^d\times \mathbb R^d)$. By \eqref{eq:lim W12}, we deduce that
$$
\int_{\bt^d \times \br^d} f(t)\,\phi\,dx\,dv=\int_{\bt^d \times \br^d} g\,\phi\,dx\,dv \qquad \text{for all}\, \, \,\phi\in C_c^\infty(\mathbb T^d\times \mathbb R^d).
$$
Thus $f(t) = g$. The uniqueness of the limit implies that in fact the whole original subsequence $f^{(r_n)}(t)$ converges to $f(t)$ weakly in $L^p(\bt^d \times \br^d)$ for $p \in [1,+\infty)$ and in $L^\infty(\bt^d \times \br^d)$ in weak$^*$ sense.

Next we show that the convergence also holds for the mass density. Since $f^{(r_n)}(t)$ converges weakly in $L^1(\bt^d \times \br^d)$, for all $\phi \in L^\infty(\bt^d)$ we have
\begin{align}
\lim_{n \to \infty}\int_{\bt^d} \rho[f^{(r_n)}(t)] \phi(x) \di x & = \lim_{n \to \infty} \int_{\bt^d \times \br^d} f^{(r_n)}(t,x,v) \phi(x) \di x \di v \\
& =  \int_{\bt^d \times \br^d} f(t,x,v) \phi(x) \di x \di v  = \int_{\bt^d} \rho_{f}(t,x,v) \phi(x) \di x .
\end{align}
In other words $\rho_{r_n}(t) \rightharpoonup \rho_f(t)$ weakly in $L^1(\bt^d)$. Since, by \eqref{rho-reg-bdd}, $\rho[f^{(r_n)}(t)]$ are uniformly bounded in $L^p(\bt^d)$ for all  $p \in [1,2]$, the convergence also holds in $L^p(\bt^d)$ in weak sense for $p \in [1, 2]$. In particular,
\be
\sup_{t \in [0,T]} \|\rho_f(t)\|_{L^p(\bt^d)} \leq \liminf_n \| \rho[f^{(r_n)}(t)]\|_{L^p(\bt^d)} .
\ee
We deduce that
\begin{align}
\sup_{t \in [0,T]} \|\rho_f(t)\|_{L^{\frac{d+2}{d}}(\bt^d)} \leq C, \qquad \label{limf-rhobdd}
\sup_{t \in [0,T]} \|\rho_f(t)\|_{L^2(\bt^d)} \leq C_T.
\end{align}

\noindent Next, we prove convergence of the electric field. By \eqref{elec-reg}, for any $\alpha \in (0,1)$,
\be \label{U-bdd}
\sup_{r} \sup_{t \in [0,T]} \| U_r(t) \|_{C^{0, \alpha}(\bt^d)} \leq C[\alpha, f_0], \quad \sup_{r} \sup_{t \in [0,T]} \| \nabla \widehat U_r(t) \|_{C^{0, \alpha}(\bt^d)} \leq C[\alpha, f_0] 
\ee
which implies that $U_r(t)$ and $\nabla \widehat U_r(t)$ are equicontinuous on $\bt^d$. Moreover
\be
\sup_{r} \sup_{t \in [0,T]} \| \nabla \bar U_r(t) \|_{W^{1,2}(\bt^d)} \leq C[T, f_0].
\ee

Hence there exists a further subsequence for which $\bar U_{r_{n_k}}(t), \widehat U_{r_{n_k}}(t), \nabla \widehat U_{r_{n_k}}(t)$ converge in $C(\bt^d)$ to some $\bar U(t), \widehat U(t), \nabla \widehat U(t)$ and $\nabla \bar U_{r_{n_k}}(t)$ converges strongly in $L^2(\TT^d)$ to $\nabla \bar U(t)$.

We identify the limit $U(t) = \bar U(t) + \widehat U(t)$, by showing that it is a solution of
\be  \label{eqn:elliptic}
\Delta U(t) = e^{U(t)} - \rho_f(t) .
\ee
The elliptic equation for $U_r(t)$ in \eqref{vpme-reg} in weak form gives that for all $r$ and all $\phi \in W^{1,2} \cap L^{1}(\bt^d)$,
$$
\int_{\bt^d} \nabla U_r(t) \cdot \nabla \phi + \left (e^{U_r(t)} - \chi_r \ast \rho_{f^{(r)}(t)} \right) \phi \di x = 0 .
$$
{
The first term converges since $\nabla U_r(t)$ converges to $\nabla U(t)$ in $L^2(\TT^d)$. The second term converges by dominated convergence, since $U_r(t)$ are uniformly bounded in $C(\bt^d)$.
}
For the term involving $\chi_r \ast \rho[{f^{(r)}(t)}]$, we split
\be \label{Ur-source-split}
\int_{\bt^d} \( \chi_r \ast \rho[{f^{(r)}(t)}]- \rho_f(t) \) \phi \di x = \int_{\bt^d} \( \chi_r \ast \rho[{f^{(r)}(t)}] - \rho[{f^{(r)}(t)}] \) \phi \di x  + \int_{\bt^d}  \(\rho[{f^{(r)}(t)}] - \rho_f(t) \)\, \phi \di x .
\ee
For any $\phi \in L^{\frac{d+2}{2}}(\bt^d)$, we have
\begin{align}
\left \lvert \int_{\bt^d}  \left (\chi_r \ast \rho[{f^{(r)}(t)}] - \rho[{f^{(r)}(t)}] \right ) \phi \di x \right \rvert & =  \left \lvert \int_{\bt^d}  (\chi_r \ast \phi - \phi) \rho[{f^{(r)}(t)}] \di x \right \rvert \\
& \leq \lVert \chi_r \ast \phi - \phi \rVert_{L^{\frac{d+2}{2}}(\bt^d)} \lVert \rho[{f^{(r)}(t)}] \rVert_{L^{\frac{d+2}{d}}(\bt^d)}\\
&  \leq C \lVert \chi_r \ast \phi - \phi \rVert_{L^{\frac{d+2}{2}}(\bt^d)} .
\end{align}
The right hand side converges to zero as $r$ tends to zero by standard results on the continuity of mollification in $L^p$ spaces. For $r = r_{n_k}$, the second term of \eqref{Ur-source-split} converges to zero as $k$ tends to infinity, for all $\phi \in L^{\frac{d+2}{2}}(\bt^d)$, since $\rho[{f^{(r)}(t)}]$ converges to $\rho_f(t)$ weakly in $L^{\frac{d+2}{d}}(\bt^d)$. Hence, for all $\phi \in W^{1,2} \cap L^{\frac{d+2}{2}}(\bt^d)$,
$$
\int_{\bt^d} \nabla U(t) \cdot \nabla \phi + (e^{U(t)} - \rho_f(t)) \phi \di x = 0 .
$$
Since {$U(t) \in C \cap W^{1,2}(\bt^d)$} and $\rho_f(t) \in L^2(\bt^d)$, this extends to all $\phi \in W^{1,2}(\bt^d)$ by density of $L^{\frac{d+2}{2}}(\bt^d)$ in $L^2(\bt^d)$ . In other words $U(t)$ is indeed a weak solution of \eqref{eqn:elliptic}.

Our earlier stability estimates imply that \eqref{eqn:elliptic} has at most one solution in $L^{\infty} \cap W^{1,2}(\bt^d)$, which is therefore $U(t)$ in this case. 
Since the limit of any convergent subsequence is uniquely identified, it follows that for all $t$ we have $U_{r_n}(t) \to U(t)$ in $C(\bt^d)$ and $W^{1,2}(\TT^d)$, where $U(t)$ is the unique $L^{\infty} \cap W^{1,2}(\bt^d)$ solution of \eqref{eqn:elliptic} (that is, without passing to further subsequences).

Next we consider the convergence of the regularised electric field
$$
E_{r_n} [f^{(r_n)}(t)] = - \chi_r \ast \nabla U_{r_n}(t) .
$$
{
Since
\begin{align}
\| E_{r_n}(t) + \nabla U(t) \|_{L^2(\TT^d)} &\leq \| \chi_r \ast (\nabla U_{r_n}(t) - \nabla U(t) )\|_{L^2(\TT^d)} + \|\nabla U(t) - \chi_r \ast \nabla U(t)\| _{L^2(\TT^d)} \\
& \leq  \|\nabla U_{r_n}(t) - \nabla U(t)\|_{L^2(\TT^d)} + \|\nabla U(t) - \chi_r \ast \nabla U(t)\| _{L^2(\TT^d)},
\end{align}
it follows that $E_{r_n}$ converges to $- \nabla U(t)$ strongly in $L^2(\TT^d)$.
}

Finally, we show that $f$ is a weak solution of \eqref{vpme}. Since $f^{(r)}$ is a solution of \eqref{vpme-reg}, for any $\phi \in C^\infty_c([0,\infty) \times \bt^d  \times  \br^d )$ we have
\begin{align}
\int_{\bt^d \times \br^d} f_0(x,v) \phi(0, x, v) \di x \di v + \int_0^\infty \int_{ \bt^d \times \br^d} (\partial_t \phi + v \cdot \nabla_x \phi +  E_r(x) \cdot \nabla_v \phi) f^{(r)} \di x \di v \di t  = 0 .
\end{align}
Since $\partial_t \phi + v \cdot \nabla_x \phi  \in C^\infty_c([0,\infty) \times \bt^d  \times  \br^d )$, \eqref{f-lim} implies that for all fixed $t$, as $n$ tends to infinity,
\begin{align}
\int_{ \bt^d \times \br^d} (\partial_t \phi + v \cdot \nabla_x \phi ) f^{(r_n)} \di x \di v \to \int_{ \bt^d \times \br^d} (\partial_t \phi + v \cdot \nabla_x \phi ) f \di x \di v .
\end{align}
Since
\be
\left \lvert \int_{ \bt^d \times \br^d} (\partial_t \phi + v \cdot \nabla_x \phi ) f^{(r_n)} \di x \di v  \right \rvert \leq \| f_0 \|_{L^\infty(\bt^d \times \br^d)} \int_{\bt^d \times \br^d} \left \lvert \partial_t \phi + v \cdot \nabla_x \phi \right \rvert \di x \di v \in L^1([0,\infty)),
\ee
we deduce from the dominated convergence theorem that for all $\phi \in C^\infty_c([0,\infty) \times \bt^d  \times  \br^d )$, as $n$ tends to infinity,
\begin{align}
\int_0^\infty \int_{ \bt^d \times \br^d} (\partial_t \phi + v \cdot \nabla_x \phi ) f^{(r_n)} \di x \di v \to \int_0^\infty \int_{ \bt^d \times \br^d} (\partial_t \phi + v \cdot \nabla_x \phi ) f \di x \di v .
\end{align}

For the nonlinear term we have the estimate
\begin{align}
\left \lvert  \int_0^\infty \int_{\bt^d \times \br^d} \nabla_v \phi \cdot \left( E_r f^{(r)} + \nabla_x U f \right )\di x \di v \di t \right \rvert & \leq \left \lvert \int_0^\infty \int_{\bt^d \times \br^d} \left ( E_r + \nabla_x U \right ) \cdot \nabla_v \phi \, f^{(r)}  \di x \di v \di t \right \rvert  \\
& + \left \lvert  \int_0^\infty \int_{ \bt^d \times \br^d} - \nabla_x U \cdot \nabla_v \phi \, (f^{(r)} -  f)  \di x \di v \di t \right \rvert .
\end{align}
Now let $T>0$ be such that $\operatorname{supp}{\phi} \subset [0,T] \times \bt^d \times \br^d$. For the first term, we use the fact that $E_{r_n}$ converges to $- \nabla U$ strongly in $L^2(\TT^d)$ for each $t$. That is, for fixed $t$
\begin{align}
\int_{\bt^d \times \br^d} \left ( E_r + \nabla_x U \right ) \cdot \nabla_v \phi \, f^{(r)}  \di x \di v & \leq \int_{\RR^d} \| E_{r_n}(t) + \nabla U(t) \|_{L^2(\TT^d)} \| \nabla_v \phi(\cdot, v ) \|_{L^2(\RR^d)} \di v \\
& \leq C_\phi  \| E_{r_n}(t) + \nabla U(t) \|_{L^2(\TT^d)} \leq C(\phi, T , f_0) .
\end{align}
Then, using dominated convergence for the time integral, as $r_n \to 0$ we have
\be
\left  \lvert \int_0^\infty \int_{\bt^d \times \br^d} \left ( E_{r_n} + \nabla_x U \right ) \cdot \nabla_v \phi \, f^{(r_n)}  \di x \di v \di t \right \rvert \to 0 .
\ee
Similarly, for the second term we use that, for each $t$, since $\nabla U(t) \cdot \nabla_v \phi \in L^2(\bt^d \times \br^d)$,
\be
\left \lvert \int_{ \bt^d \times \br^d} \nabla U \cdot \nabla_v \phi \, (f^{(r_n)} -  f)  \di x \di v  \right \rvert \to 0 .
\ee
Combining this with the bound
\be
\left \lvert  \nabla U \cdot \nabla_v \phi \, (f^{(r)} -  f) \right \rvert  \leq C(T, f_0)  \, | \nabla_v \phi | \in L^1([0, \infty) \times \bt^d \times \br^d ),
\ee
which follows from \eqref{vpme-reg-Lp} and \eqref{elec-reg}, we conclude that, as $n$ tends to infinity,
\be
\int_0^\infty \int_{\bt^d \times \br^d} \nabla_v \phi \cdot E_{r_n} \, f^{(r_n)} \di x \di v \di t \to - \int_0^\infty \int_{\bt^d \times \br^d} \nabla_v \phi \cdot \nabla U f \di x \di v \di t .
\ee
Hence, for all $\phi \in C^\infty_c([0,+\infty) \times \bt^d \times \br^d)$,
\be
\int_{\bt^d \times \br^d} f_0(x,v) \phi(0, x, v) \di x \di v + \int_0^\infty \int_{\bt^d \times \br^d)} (\partial_t \phi + v \cdot \nabla_x \phi -  \nabla U (x) \cdot \nabla_v \phi) f \di x \di v \di t = 0 .
\ee
Thus $f$ is a weak solution of \eqref{vpme}.

{

We next show that the bounds on the velocity moments pass to the limit. Recall that there exists a constant $C_{m_0,0}$ independent of $r$ such that for all $r$,
\be
\sup_{t \in[0,T]} \int_{\TT^d \times \RR^d} |v|^{m_0} f^{(r)}(t,x,v) \di x \di v \leq C_{m_0,0} .
\ee
Then, since $f^{(r)}$ converges to $f$ weakly in $L^1(\TT^d \times \RR^d)$, for all $R > 0$ we have
\be
\int_{\TT^d \times \RR^d} |v|^{m_0} \mathbf{1}_{|v| \leq R} \, f(t,x,v) \di x \di v = \lim_{r \to 0} \int_{\TT^d \times \RR^d} |v|^{m_0} f^{(r)}(t,x,v)  \mathbf{1}_{|v| \leq R} \, \di x \di v \leq C_{m_0,0} .
\ee
Therefore, letting $R\to \infty,$ 
\be
\int_{\TT^d \times \RR^d} |v|^{m_0} f(t,x,v) \di x \di v \leq C_{m_0,0} .
\ee
Finally, we show that the regularised energy functionals $\mc{E}_r[f^{(r)}]$ converge to the energy functional $\mc{E}[f]$. Since
\be
\sup_r \sup_{t \in[0,T]} \int_{\TT^d \times \RR^d} |v|^{m_0} f^{(r)}(t,x,v) \di x \di v \leq C_{m_0,0} ,
\ee
from the weak $L^1$ convergence of $f^{(r)}$ to $f$ it follows that
\be
\int_{\TT^d \times \RR^d} |v|^2 f(t,x,v) \di x \di v = \lim_{r \to 0} \int_{\TT^d \times \RR^d} |v|^2 f^{(r)}(t,x,v) \di x \di v .
\ee
Since the convergence of $E_r$ to $- \nabla U$ occurs strongly in $L^2(\TT^d \times \RR^d)$ for all $t$, it follows that
\be
\int_{\TT^d} |\nabla_x U(t)|^2 \di x = \lim_{r \to 0} \int_{\TT^d} |\nabla_x U_r(t)|^2 \di x .
\ee
Lastly, since $U_r(t)$ converges to $U(t)$ in $C(\TT^d)$, we have
\be
\int_{\TT^d} U(t) e^{U(t)} \di x = \lim_{r \to 0} \int_{\TT^d} U_r(t) e^{U_r(t)} \di x
\ee
We conclude that
\be
\mc{E}[f(t)] = \lim_{r \to 0} \mc{E}_r [f^{(r)}(t)] =  \lim_{r \to 0} \mc{E}_r [f_0] = \mc{E} [f_0] .
\ee
}
\end{proof}

\subsection{Global Well-posedness}
{
Lastly, we complete the proof of the main result, Theorem~\ref{thm:main}, showing global well-posedness for the VPME system given initial data decaying sufficiently fast at infinity.

\begin{proof}
By Lemma~\ref{lem:exists-weak-solution}, under the assumptions of Theorem~\ref{thm:main} there exists a weak solution $f$ of the VPME system \eqref{vpme}. We will show that $\rho_f \in L^{\infty}_{\text{loc}}([0, + \infty) ; L^\infty(\TT^d))$, so that uniqueness will then follow from Theorem~\ref{thm:unique}.

We first derive estimates on the electric field $E$.
Since the initial datum $f_0$ is assume to have a finite velocity moment of order $m_0$, the solution also has moments of this order: for all $T>0$, there exists a constant $C_T>0$ such that
\be
\sup_{t \in [0,T]} \int_{\TT^d \times \RR^d} |v|^{m_0} f(t,x,v) \di x \di v \leq C_T .
\ee
We apply the interpolation estimate from Lemma~\ref{lem:mom-interpolation} to deduce an estimate on $\rho_f$. Since $m_0 > d(d-1)$, we obtain
\be
\sup_{t \in [0,T]} \| \rho_f \|_{L^{d + \alpha}(\TT^d)} \leq C_T
\ee
for some $\alpha > 0$. By regularity estimates for the Poisson equation and Sobolev inequalities, $\bar E$ is H\"{o}lder continuous with the following estimate for some $\gamma > 0$:
\be
\sup_{t \in [0,T]} \| E(t,\cdot) \|_{C^{0, \gamma}(\TT^d)} \leq C_T .
\ee
We also have the uniform $L^{\frac{d+2}{d}}(\TT^d)$ estimate from the conservation of energy:
\be
\sup_{t\in [0,T]} \| \rho_f(t, \cdot) \|_{L^{\frac{d+2}{d}}(\TT^d} \leq C_0 .
\ee
Proposition~\ref{prop:regU} then implies that
\be
\| \widehat E \|_{C^{0,\gamma}(\TT^d)} \leq C_0 .
\ee
We deduce that $E \in L^\infty(\TT^d)$ with $\sup_{t \in[0,T]} \| E(t, \cdot) \|_{L^\infty(\TT^d)} \leq C_T$. 
We then propagate the uniform decay estimate on $f$ by using the characteristic flow. Namely,
\be
f(t,x,v) = f_0 \left ( X(0; t,x,v), V(0;,t,x,v) \right ) \leq \frac{C_0}{1 + |V(0; t,x,v)|^{k_0}} .
\ee
Since $E$ is bounded in $L^\infty$, we obtain an estimate of the form
\be
f(t,x,v) \leq \frac{C_0}{1 + \(|v| - C_T t\)_+^{k_0}} 
\ee
Since $k_0 > d$, this bound can be integrated with respect to $v \in \RR^d$ to obtain
\be
\rho_{f}(t,x) \leq C(T, f_0) ,
\ee
which completes the proof.

\end{proof}
}

\medskip

{\it Acknowledgements:} This work was supported by the UK Engineering and Physical Sciences Research Council (EPSRC) grant EP/L016516/1 for the University of Cambridge Centre for Doctoral Training, the Cambridge Centre for Analysis; and partly funded by the European Research Council (ERC) under the European Union's Horizon 2020 research and innovation programme (grant agreement No 726386).
The authors wish to acknowledge Maxime Hauray, Cl\'ement Mouhot, and Claude Warnick for useful comments on a preliminary version of this manuscript.

\end{document}